\documentclass{amsart}

\usepackage{latexsym,amsmath,amssymb,amsfonts,amscd,graphics,appendix,amsxtra}
\usepackage[pdftex]{graphicx}
\usepackage[mathscr]{eucal}
\usepackage[all]{xy}
\usepackage[normalem]{ulem}
\usepackage[]{hyperref}

\numberwithin{equation}{section}
\theoremstyle{plain}
\newtheorem{theorem}[equation]{Theorem}
\newtheorem*{mainthm}{Main Theorem}

\newtheorem{lemma}[equation]{Lemma}

\newtheorem{proposition}[equation]{Proposition}
\newtheorem{corollary}[equation]{Corollary}

\theoremstyle{remark}
\newtheorem{remark}[equation]{Remark}

\newtheorem*{remintro}{Remark}
\theoremstyle{definition}
\newtheorem{definition}[equation]{Definition}
\newtheorem{notation}[equation]{Notation}

\newtheorem{example}[equation]{Example}

\def\Pic{\operatorname{Pic}}

\def\Sym{\operatorname{Sym}}

\newcommand{\bP}{\mathbb{P}}
\newcommand{\bA}{\mathbb{A}}

\newcommand{\bQ}{\mathbb{Q}}
\newcommand{\bZ}{\mathbb{Z}}
\newcommand{\bF}{\mathbb{F}}

\newcommand{\bC}{\mathbb{C}}

\newcommand{\fH}{\mathfrak{H}}

\newcommand{\calA}{\mathcal{A}}

\newcommand{\calF}{\mathcal{F}}

\newcommand{\calH}{\mathcal{H}}
\newcommand{\calN}{\mathcal{N}}

\newcommand{\calM}{\mathcal{M}}
\newcommand{\calO}{\mathcal{O}}
\newcommand{\calL}{\mathcal{L}}
\newcommand{\calP}{\mathcal{P}}

\newcommand{\calX}{\mathcal{X}}

\newcommand{\calD}{\mathcal{D}}

\newcommand{\SL}{\mathrm{SL}}

\newcommand{\sX}{\mathscr{X}}

\newcommand{\Gm}{\mathbb{G}_m}
\newcommand{\Es}{\widetilde{E}_7}
\newcommand{\Ee}{\widetilde{E}_8}
\newcommand{\Er}{\widetilde{E}_r}

\newcommand{\Proj}{\mathrm{Proj}}

\newcommand{\Gr}{\mathrm{Gr}}

\newcommand{\gquot}{/\!\!/}

\newcommand{\oD}{\overline{D}}
\newcommand{\oV}{\overline{V}}
\newcommand{\oL}{\overline{L}}
\newcommand{\oH}{\overline{H}}
\newcommand{\II}{\textrm{II}}
\newcommand{\III}{\textrm{III}}

\title[KSBA for degree $2$ $K3$ pairs]{The KSBA compactification for the moduli space of degree two $K3$ pairs} 
\author[R. Laza]{Radu Laza}
\address{Stony Brook University, Department of Mathematics,  Stony Brook, NY 11794}
\email{rlaza@math.sunysb.edu}

\thanks{The author was partially supported by NSF grant DMS-0968968 and a Sloan Fellowship}

\begin{document}
\bibliographystyle{amsalpha}

\begin{abstract} Inspired by the ideas of the minimal model program,   Shepherd-Barron, Koll\'ar, and Alexeev have constructed a geometric compactification for the moduli space of surfaces of log general type. In this paper, we discuss one of the simplest examples that fits into this framework:  the case of pairs $(X,H)$ consisting of a degree two $K3$ surface $X$   and an ample divisor $H$. Specifically, we construct and describe explicitly a  geometric compactification $\overline \calP_2$  for the moduli of degree two $K3$ pairs. This compactification has a natural forgetful map to the Baily--Borel compactification of the moduli space $\calF_2$ of degree two $K3$ surfaces. Using this map and the modular meaning of $\overline \calP_2$, we obtain a better understanding  of the geometry of the  standard compactifications  of $\calF_2$. 
\end{abstract}

\maketitle
\section*{Introduction}
The search for geometric compactifications for moduli spaces is one of the central problems in algebraic geometry. After the successful constructions of compactifications for the moduli spaces of curves (Deligne--Mumford), and   abelian varieties (Mumford, Namikawa, Alexeev, and others), a case that attracted a great deal of interest was that of polarized $K3$ surfaces (e.g. \cite{fmbook}). Similar to the case of abelian varieties,  the moduli space of polarized $K3$ surfaces is a locally symmetric variety and as such it has several compactifications, the most commonly studied being the Baily-Borel and  toroidal compactifications. Unfortunately, very little is known about the geometric meaning of those. The best understood situation is that of low degree $K3$ surfaces where algebraic constructions for the moduli space are available via GIT.  Namely, for degree $2$ (and similarly for degree $4$), Shah  constructed a compactification $\widehat \calM$ for the moduli of  degree $2$ $K3$ surfaces which has several good properties (see Thm. \ref{thmshah2}). For instance, $\widehat \calM$ is an Artin stack with weak modular meaning (in the sense of GIT): $\widehat \calM$ parameterizes degenerations of $K3$ surfaces that are Gorenstein and have at worse semi-log-canonical singularities.

The space $\widehat{\calM}$ was constructed by Shah \cite{shah} as a partial Kirwan desingularization of the GIT quotient $\overline \calM$ for sextic curves (see also \cite{kirwanlee}). Alternatively, for any degree,  the  moduli space of polarized $K3$ surfaces is isomorphic to a locally symmetric variety $\calD/\Gamma_d$. Then,  the space $\calD/\Gamma_d$ has a natural compactification, the Baily--Borel compactification $(\calD/\Gamma_d)^*$. 
For degree $2$, as shown by Looijenga \cite{looijengavancouver}, $\widehat{\calM}$ is a small partial resolution of $(\calD/\Gamma_2)^*$, and is in fact a semi-toric compactification in the sense of  \cite{looijengacompact} (see Thm. \ref{thmlooijenga}). Thus, $\widehat{\calM}$  has a dual description which gives complementary information: the GIT construction provides some  geometric meaning to the boundary, and, on the other hand, the semi-toric construction gives a rich structure, which can be further exploited in applications.  Arguably $\widehat \calM$ is the ``best'' compactification for the moduli space $\calF_2$ of degree two $K3$ surfaces known at this point.

The issue is that $\widehat \calM$ is not modular in the usual sense: it fails to be separated at the boundary. While one might hope that some toroidal compactification $\overline{\calD/\Gamma_2}^\Sigma$ (refining $\widehat \calM$ and $(\calD/\Gamma_2)^*$) would give a modular compactification for $\calF_2$, as in the case of abelian varieties (see \cite{namikawabook}, \cite{alexeevabelian}), this is not known and seems out-of-reach (see however \cite{olsson} and Remark \ref{remtoroidal}). In this paper, we go in a different direction. Namely, we modify the moduli problem and construct a modular compactification $\overline{\calP}_2$ of the corresponding moduli space, which admits a forgetful map $\overline{\calP}_2\to(\calD/\Gamma_2)^*$ (generically a $\bP^2$-fibration). In other words, we obtain a fibration with modular meaning over some compactification of $\calF_2$. We note that $\overline{\calP}_2$ sheds further light on the geometric meaning on the standard compactifications (e.g. GIT, Baily-Borel) of $\calF_2$ and we expect it to play an important role in the elusive search for a geometric compactification for the moduli of  $K3$ surfaces.

Concretely, we consider the moduli space $\calP_2$ of pairs $(X,H)$ consisting of a degree $2$ $K3$ surface and an ample divisor of degree $2$. There is a natural forgetful map $\calP_2\to \calF_2$ given by $(X,H)\to (X,\calO_X(H))$, that makes $\calP_2$ a $\bP^2$-bundle over the moduli space of degree $2$ $K3$ surfaces. We compactify $\calP_2$ using the framework introduced by Koll\'ar--Shepherd-Barron \cite{ksb} and Alexeev \cite{alexeevpairs0} (called KSBA in what follows) and the $\epsilon$-coefficient approach pioneered by Hacking \cite{hacking}. The main idea of this approach is to view a degree $2$ pair as a log general type pair  $(X,\epsilon H)$ and to compactify by allowing stable pairs (i.e. require $(X,\epsilon H)$ to have slc singularities and $H$ to be ample). Then, a geometric compactification for $\calP_2$ exists by general principles in the minimal model program (MMP).  In fact, the same is true for all degrees, and thus one obtains geometric compactifications $\overline \calP_d$ for all degrees $d\in 2\bZ_+$ (see  Cor. \ref{pdthm}). The issue is that it is very difficult to understand $\overline \calP_d$ directly. The main result of the paper is to construct $\overline \calP_2$ explicitly and to describe the boundary pairs. We summarize the main result as follows:

\begin{mainthm}
Let $\calF_2$ and $\calP_2$ be the moduli space of degree two $K3$ surfaces and degree two pairs respectively. There exists a geometric compactification $\overline{\calP}_2$ of $\calP_2$ parameterizing stable degree $2$ pairs (Def. \ref{defstable}) and a natural map $\overline{\calP}_2\to (\calD/\Gamma_2)^*$ to the Baily-Borel compactification extending the forgetful map $\calP_2\to \calF_2$. Furthermore, there exist six irreducible boundary components for $\overline{\calP}_2$  of dimensions: $3$, $4$, $10$, $12$, $13$, and $19$ respectively. The geometric meaning of these components is described in Table \ref{table1} (see Theorems \ref{thmtype2}  and \ref{thmtype3} and Table \ref{table2} for further details). 
\end{mainthm}
\begin{table}[htb!]
\renewcommand{\arraystretch}{1.3}
\begin{tabular}{|l|l|l|l|l|}
\hline
&Description (generic point)& dim& Type II Case &  Type III\\
\hline\hline
1& $X=V_1\cup_E V_2$, $V_i\cong\bP^2$&3&$A_{17}$& $E$ nodal\\
\hline
2& $X=V_1\cup_E V_2$, $V_i$ are deg. $1$ del Pezzos&19&$E_8^{2}+A_1$ (A)& $E$ nodal\\
\hline
3&$X^\nu$ is a quadric in $\bP^3$, double curve $E$&4&$D_{16}+A_1$& $E$ nodal, or\\
&&&&$E=C_1\cup C_2$\\
\hline
4&$X^\nu$  deg. 2 del Pezzo, double curve $E$&10&$E_7+D_{10}$ (A)&$E$ nodal\\
\hline
5&$X$ is rational with an $\widetilde E_8$ singularity&12&$E_8^{2}+A_1$ (B)& $T_{2,3,7}$\\
\hline
6&$X$ is rational with an $\widetilde E_7$ singularity&13&$E_7+D_{10}$ (B)& $T_{2,4,5}$\\
\hline
\end{tabular}
\vspace{0.2cm}
\caption{Boundary components of $\overline \calP_2$}\label{table1}
\end{table}
We recall that the Baily-Borel compactification $(\calD/\Gamma_2)^*$ is obtained by adding four rational curves to  $\calF_2$ (see Thm. \ref{thmbb} and Fig. \ref{fig4}). Each of the six boundary components of $\overline \calP_2$ will map to one of the four Baily-Borel boundary components, giving them  a fibration structure over $\bP^1$. For instance, the three dimensional boundary component of $\overline \calP_2$ corresponding to the first case of Table \ref{table1} is a $\bP^2$-fibration over $\bP^1$ (the closure of the Type II Baily--Borel boundary component $\II_{A_{17}}$). For further details see the following remark and Sections \ref{secttype2} and \ref{secttype3}. 

\begin{remintro}
Here we make some comments on the content of Table \ref{table1}. The boundary components are labeled by the cases of Proposition \ref{classifydeg2}. The second column describes the generic stable pair $(X,H)$ parameterized by a boundary component. The class of the polarizing divisor $H$ is easily determined in each case, and we omit it from the description.  In the table, $E$ refers to an  anticanonical divisor on some (normalized) component of $X$. The map sending a boundary component in $\overline \calP_2$ to a Baily-Borel boundary component (which is isomorphic to $\bP^1$) is given by the $j$-invariant of $E$. The division into Type II (i.e. $E$ smooth) cases  is discussed in Section \ref{secttype2}. The column labeled Type III describes the generic degeneracy condition to get a Type III case (see Section \ref{secttype3}). Note that in case (3) there are two (codimension $1$) possibilities for the degenerations of $E$: either a nodal quartic curve in $\bP^3$ or a union of two hyperplane sections of a quadric in $\bP^3$. 
\end{remintro}

Our approach to understanding $\overline \calP_2$ is to relate this space to a GIT quotient for pairs. Specifically, we first construct a GIT quotient $\widehat{\calP}_2$  and a natural forgetful map $\widehat{\calP}_2\to \widehat{\calM}$ (see Thm. \ref{thmgitpair}) by including the GIT analysis of Shah \cite{shah} into a larger VGIT problem that takes into account the polarization divisor as well. This VGIT set-up is quite similar to that of  \cite{raduthesis}. To get an idea of the set-up and of why considering divisors instead of line bundles is relevant, we recommend the reader to see first the example  discussed in \S\ref{sectexample}.

The GIT space $\widehat \calP_2$ is not the same as $\overline \calP_2$, but they agree over the stable locus in $\widehat \calP_2$. We show $\overline \calP_2$ is a flip of $\widehat \calP_2$ along the semi-stable locus (see Thm. \ref{thmflip}).  The main point in comparing the GIT and KSBA compactifications  is a good understanding of 
the GIT boundary pairs and the results on linear systems on anticanonical pairs of Friedman \cite{friedmanlin} and Harbourne \cite{har1,har2} (see esp. Prop. \ref{classifydeg2}).

Our paper builds on the work on $K3$ surfaces of Shah \cite{shah,shahinsignificant}, Looijenga \cite{looijengavancouver,looijengacompact,looijengarational}, Friedman and Scattone \cite{friedmanannals,friedmanscattone,scattone}, and on  the work on compactifications of Koll\'ar \cite{kollar}, Shepherd-Barron \cite{sbnef,sbpolarization}, \cite{ksb}, Alexeev \cite{alexeevpairs0}, and Hacking \cite{hacking}. We also note that some discussion of degenerations of degree $2$ $K3$ surfaces from the perspective of the minimal model program was done recently by Thompson \cite{thompson} (see Thm. \ref{thmthompson}). The main difference to our paper is that \cite{thompson} never keeps track of the polarizing divisor $H$, and consequently it is not possible to fit the degenerations occurring in \cite{thompson} into a proper and separated moduli stack. We believe that one of the main contributions this paper provides for the general theory of moduli is to show concretely the importance of working with log general type: by considering polarizing divisors instead of polarizations, the boundary points are  naturally separated and fit into a moduli space. The example of \S\ref{sectexample} clearly illustrates this point in a simple case. Related to this example, we note that the moduli of weighted pointed curves considered by Hassett \cite{hassettweighted} is a $1$-dimensional analogue (esp. for genus $1$) of the moduli problem considered here. Finally, Hacking--Keel--Tevelev \cite{hkt} is another application of the KSBA approach to compactifying moduli spaces of special classes of surfaces (in loc. cit. del Pezzo).

We close with some remarks about the general degree $d$ case. First, a very similar analysis (involving GIT) can be carried out for other low degree cases. On the other hand, in general, the results of Section \ref{sectksba} establish the existence of a geometric compactification $\overline \calP_d$ for the moduli of degree $d$ $K3$ pairs. By Hodge theoretic considerations (see \cite{shahinsignificant}, \cite{kovacs}, and \cite{usui}),  we also expect that this compactification  maps to the Baily-Borel compactification (i.e. $\overline \calP_d\to (\calD/\Gamma_d)^*$). Then, the results of Section \ref{sectpolanti} give a procedure for identifying the essential components (i.e. the ``$0$-surfaces'') of the central fiber  in a degree $d$ degeneration. In principle, for a given degree $d$,  these techniques would allow one to identify the boundary components in $\overline \calP_d$.   However, as the degree increases, the number of cases in a classification of $0$-surfaces (analogue to Prop. \ref{classifydeg2}) and the number of gluing of these $0$-surfaces will grow very fast (roughly proportional to the number of partitions of $d$), making an explicit classification unfeasible for large $d$. Finally, we note that the GIT approach  (for small $d$) not only helps classify the boundary cases, but also gives a lot of structure to the fibration $\overline \calP_d\to (\calD/\Gamma_d)^*$.

We are also aware of some partial results and general approaches to the study of $\overline \calP_d$ of other researchers (e.g. \cite{ghk}). While we are considering only the degree two case here, our study is the first complete analysis of a geometric compactification for $K3$ pairs and one of the first in the KSBA framework for log general type surfaces (see also \cite{hkt}). We believe that our study is relevant to the general $\overline \calP_d$ case and  to the original compactification problem for $K3$ surfaces. 

\subsection*{Organization} In section \ref{sectreview}, we review the standard compactifications for moduli of degree $2$ $K3$s and discuss the space  $\widehat \calM$. This material is standard, but rather scattered  throughout the literature. Then, in section \ref{sectksba}, we introduce the KSBA compactification (based on \cite{sbpolarization}, \cite{ksb}, \cite{alexeevpairs0}, and \cite{hacking}) and establish the existence of a modular compactification $\overline \calP_d$. Next, in Section \ref{sectpolanti}, we review and adapt  some results on linear systems on anticanonical pairs of Friedman  and Harbourne.

The actual construction of $\overline{\calP}_2$ starts in  Section \ref{sectvgit}, where we introduce  the VGIT problem (generalizing \cite{shah} to $K3$ pairs) and discuss the space $\widehat \calP_2$. Then,  in Section \ref{sectslclimit},  we compare the GIT compactification $\widehat \calP_2$ with the KSBA compactification $\overline \calP_2$ for the moduli of degree $2$ $K3$ pairs.  Finally, in Sections \ref{secttype2} and \ref{secttype3}, we discuss in some detail the classification of the Type II and Type III degenerations respectively. Here, we also discuss the connection to the standard compactifications (GIT, Baily-Borel, or partial toroidal) of $\calF_2$.

\subsection*{Acnowledgement} The idea of considering a moduli of stable pairs as an alternative solution to the compactification problem for $K3$ surfaces is widely discussed among the experts in the field. We have benefited from long term discussions with V. Alexeev, R. Friedman, P. Hacking, B. Hassett, S. Keel, and E. Looijenga. We are also grateful to R. Friedman, P. Hacking and J. Koll\'ar for some specific comments on an earlier draft.

%%%%%%%%%%%%%%%%%%%%%%%%%%%%%%%%%%%%%%
%%% Sect1 : Compactifications of F_2
%%%%%%%%%%%%%%%%%%%%%%%%%%%%%%%%%%%%%%
\section{Review of the standard compactifications of $\calF_2$}\label{sectreview}
In this section we review some facts about the moduli space $\calF_2$ of degree $2$ $K3$ surfaces and its compactifications. While all the results here are  well known (see esp. Shah \cite{shah},  Looijenga \cite{looijengavancouver}, Friedman \cite{friedmanannals}, and Scattone \cite{scattone}), the presentation is somewhat new and adapted to the subsequent needs of the paper.

%%% Baily-Borel Compactification
\subsection{The Baily-Borel compactification}\label{sectbb} 
In general,  the  moduli space $\calF_d$ of $K3$ surfaces of degree $d$  is isomorphic to a locally symmetric variety $\calD/\Gamma_d$, where $\calD$ is a $19$-dimensional Type IV domain and $\Gamma_d$ is an arithmetic group acting on $\calD$. Namely, $\calD\cong \{\omega\in \bP(\Lambda_d\otimes_\bZ \bC) \mid \omega.\omega=0, \ \omega.\bar{\omega}>0\}_0$ and $\Gamma_d$ is a subgroup of finite index in $\calO(\Lambda_d)$, where $\Lambda_d\cong \langle -d\rangle\oplus E_8^{\oplus 2}\oplus U^{\oplus 2}$ is the primitive middle cohomology of a degree $d$ $K3$ surface. By the Baily-Borel theory, the space $\calD/\Gamma_d$ is a quasi-projective algebraic variety and admits a projective compactification $(\calD/\Gamma_d)^*$. For Type IV domains, the Baily-Borel compactification  $(\calD/\Gamma_d)^*$ is quite small: topologically, it is obtained by adding points ({\it Type III components}) and curves ({\it Type II components}), which are quotients of the upper half space $\fH$ by modular groups. 

The Baily-Borel compactifications for the moduli spaces of $K3$ surfaces were analyzed by Scattone \cite{scattone}. In particular, for the degree $2$, the following holds:
\begin{theorem}[Scattone]\label{thmbb}
The boundary of $\calF_2^*=(\calD/\Gamma_2)^*$ consists of four curves (the closures of the Type II components) meeting in a single point (the unique Type III component). Furthermore, each Type II component is isomorphic to $\fH/\SL(2,\bZ)$. 
\end{theorem}
\begin{proof}
\cite[\S6.2]{scattone} and \cite[\S5.7, esp. Fig. 5.5.7]{scattone} for the second statement.
\end{proof}

\begin{remark}\label{labelbb}
The Type II components are in one-to-one correspondence with the rank $2$ isotropic sublattices $E$ of $\Lambda_d$ modulo $\Gamma_d$. Then, $E^\perp_{\Lambda_d}/E$  is a negative definite rank $18$ lattice and a basic arithmetic invariant of $E$ (and of the corresponding  Type II component).  The subroot lattice $R$ contained in $E^\perp_{\Lambda_d}/E$ is another (coarser) arithmetic invariant. In many cases (e.g. degree $2$), $R$ uniquely determines the isometry class of $E$. Consequently, it is customary to label the Type II components by the root lattice $R$. For degree $2$, the four Type II components correspond to the root lattices $2E_8+ A_1$, $E_7+D_{10}$, $D_{16}+ A_1$, and $A_{17}$ respectively (see Figure \ref{fig4}).
\end{remark}

%%% GIT for plane sextics
\subsection{The GIT compactification}\label{sectshah}
For low degree $K3$ surfaces (e.g. $d\le 8$), an alternative (purely algebraic) construction for the moduli space $\calF_d$ can be done via GIT. Additionally, GIT produces a compactification with some weak geometric meaning. Here, we review the results of Shah \cite{shah} for degree $2$ $K3$ surfaces. The connection to the Baily-Borel compactification is discussed in \S\ref{sectlooijenga} below. 
 
 A generic $K3$ surface of degree $2$ is a double cover of $\bP^2$ branched along a plane sextic. Thus, a first approximation of the moduli space  $\calF_2$  of the degree $2$ $K3$ surfaces is the GIT quotient $\overline \calM:=\bP H^0(\bP^2,\calO_{\bP^2}(6))\gquot \SL(3)$ for plane sextics. This GIT quotient was described by Shah \cite[Thm. 2.4]{shah}.
\begin{theorem}[Shah]\label{thmshah1}
Let $\overline{\calM}$ be the GIT quotient of plane sextics.
\begin{itemize}
\item[(1)] A sextic with ADE singularities is GIT stable. Thus, there exists an  open subset $\calM\subset \overline \calM$, which is a coarse moduli space for sextics with ADE singularities (or equivalently non-unigonal degree $2$ $K3$ surfaces). 
\item[(2)] $\overline \calM\setminus \calM$ consists of $7$ strata (irreducible, locally closed, disjoint subsets):
\begin{itemize}
\item[](Type II)  $Z_1$, $Z_2$, $Z_3$, $Z_4$ of dimensions $2$, $1$, $2$, and $1$ respectively (with $Z_i$ corresponding to case $II(i)$ of \cite[Thm. 2.4]{shah});
\item[](Type III) $\tau$, and $\zeta$ of dimensions $1$ and $0$ (cf. III(1) and III(2) of \cite[Thm. 2.4]{shah});
\item[](Type IV) a point $\omega$ (cf. IV of \cite[Thm. 2.4]{shah}).  
\end{itemize}
\item[(3)] The following is a complete list of adjacencies among the boundary strata:
\begin{itemize}
\item[a)] $\zeta\in \overline{Z}_i$ for all $i\in\{1,\dots,4\}$;
\item[b)] $\overline{\tau}= \overline{Z}_1\cap \overline{Z}_3$;
\item[c)] $ \overline{\tau}=\tau\cup \{\zeta\}\cup\{\omega\}$. 
\end{itemize}
(see Figure \ref{fig3}).
\end{itemize}
\end{theorem}

\begin{remark}
Each point of a boundary strata corresponds to a unique minimal orbit. The singularities of $\overline {\calM}$ along the boundary strata depend on the stabilizers of these minimal orbits. For our situation, we have the following:
\begin{itemize}
\item[i)] The points parameterized by $Z_3$ and $Z_4$ are stable points. In particular, $\overline{\calM}$ has finite quotient singularities along these strata.
\item[ii)] The stabilizers of closed orbits parameterized by $Z_1, Z_2$, and $\tau$ are, up to finite index, $\bC^*$. 
\item[iii)] The stabilizer of the closed orbit parameterized by $\zeta$ (equation $(x_0x_1x_2)^2$) is the standard diagonal $2$-torus.
\item[iv)] The stabilizer of the closed orbit parameterized by $\omega$ (equation $(x_0x_2-x_1^2)^3$) is $\SL(2)$. 
\end{itemize} 
In particular, note that $\overline \calM$ has 
toric singularities everywhere except the point $\omega$. 
\end{remark}

As noted above, the space $\calM$ is a moduli space of curves with ADE singularities. The boundary $\overline \calM\setminus \calM$ is not strictly speaking a GIT boundary, but a boundary of non-ADE singularities. Shah has noted that except for the curves corresponding to the point $\omega$ the singularities that occur are ``{\it cohomologically insignificant singularities}" (see \cite{shahinsignificant}). In modern language, the cohomologically insignificant singularities are {\it du Bois singularities} (compare  \cite{steenbrinkinsignificant}). In the situation considered here, two dimensional hypersurfaces, these singularities are the same as the semi-log-canonical (slc) singularities of Koll\'ar--Shepherd-Barron \cite{ksb} (see also \cite{kk} and \cite{kovacs} for a more general discussion). Rephrasing the analysis of Shah (esp. \cite[Thm. 3.2]{shah}) in modern language, we get the following key result:

\begin{proposition}\label{propslc}
Let $C$ be a plane sextic, and $X$ the double cover of $\bP^2$ branched along $C$ (not necessarily normal). Then, $X$ is slc iff $C$ is GIT semistable and the closure of the orbit of $C$ does not contain the orbit of the triple conic. 
\end{proposition}
\begin{proof}
Assume first that $X$ is slc. This is equivalent to $(\bP^2, \frac{1}{2} C)$ is a log canonical pair. Then, $C$ is GIT semistable by \cite{kimlee} and \cite[\S10]{hacking}.

Conversely, assume $C$ is GIT semistable and that its orbit closure does not contain the triple conic. By the semi-continuity of the log canonical threshold, we can assume without loss of generality that  the orbit of $C$ is closed. An inspection of the  list of Shah \cite[Thm. 2.1]{shah} shows that the non-ADE singularities of $C$ are either isolated singularities of type  $\widetilde E_r$ (for $r=7,8$) or $T_{2,q,r}$,  or non-isolated singularities that lead to normal crossings, pinch points, or degenerate cusp singularities  for  the double cover $X$. The conclusion follows (e.g. see \cite[Thm. 4.21]{ksb}).

Finally, the triple conic gives a surface $X$  which does not have slc singularities. It remains to see that the same is true for semistable curves $C$ that degenerate to the triple conic. Such a curve $C$ is of type $V((x_0x_2+x_1^2)^3+f_6(x_0,x_1,x_2))$, 
where $f_6$ is a  degree $6$ polynomial which has negative degree with respect to the weights $(1,0,-1)$. Passing to affine coordinates $x,y$ around $(1,0,0)$ and after the change of coordinates $y'=y+x^2$, we get  $C$ to be given by
$$(y')^3+f_6(1,x,y'-x^2)=(y')^3+\alpha x^7+\textrm{h.o.t.},$$
where the higher order terms are with respect to the weights $\frac{1}{3}$ and $\frac{1}{7}$ for $y'$ and $x$ respectively. If $\alpha\neq 0$,  $(y')^3+\alpha x^7$ defines a singularity of type $E_{12}$ in Arnold's classification (cf. \cite[\S16.2.$7_2$]{agzv1}). Since this is a quasi-homogeneous singularity,   the log canonical threshold does not depend on the higher order terms.  We get that $X$ is not log canonical. By semi-continuity, the same is true if $\alpha=0$. 
\end{proof}

%%% Kirwan blow-up (The Unigonal Case)
\subsection{The blow-up of the point $\omega\in\overline{\calM}$}\label{sectblowup} 
By Mayer's Theorem, a degree $2$ linear system $|H|$ on a $K3$ surface $X$ is of one of the following types: 
\begin{itemize}
\item[(NU)] (Hyperelliptic case) $|H|$ is base-point-free, in which case $X$ is  a double cover of $\bP^2$ branched along a plane sextic $C$ with at worst ADE singularities.
\item[(U)] (Unigonal case) $|H|$ has a base-curve $R$, then $H=2E+R$, where $E$ is elliptic and $R$ smooth rational. The free part of $|H|$ (i.e. $2E$) maps $X$ to a plane conic, and gives an elliptic fibration on $X$. On the other hand, $|2H|$ is base-point-free and maps $X$ two-to-one to  $\Sigma_4^0\subset \bP^5$, where $\Sigma_4^0$ is the cone over the rational normal curve in $\bP^4$.  The map $X\to \Sigma_4^0$ is ramified at the vertex and in a degree $12$ curve $B$, which does not pass through the vertex.  The curve $B$ has at worst ADE singularities. 
\end{itemize}
As discussed above, all degree two $K3$ surfaces of Type (NU) correspond to stable points of $\overline \calM$. On the other hand, all the surfaces of Type (U) are mapped to the point $\omega\in \overline{\calM}$. The blow-up $\widehat \calM$ of $\omega$ will introduce all the unigonal surfaces and will give a compactification for $\calF_2$. More precisely, we restate the main result of Shah \cite{shah} as follows:

\begin{theorem}[Shah]\label{thmshah2}
The Kirwan blow-up $\widehat{\calM}$ of the point $\omega\in \overline{\calM}$ gives a projective compactification of the moduli space $\calF_2$ of degree two $K3$ surfaces.  The boundary strata of $\calF_2\subset \widehat{\calM}$ are strict transforms of the boundary strata of $\overline{\calM}$ (compare Theorem \ref{thmshah1} and see Figures 
\ref{fig3} and  \ref{fig4}). Furthermore, the boundary points of $\widehat{\calM}$ correspond (in the sense of GIT) to degenerations of $K3$ surfaces of degree $2$ that are double covers of $\bP^2$ or $\Sigma_4^0$ and have at worst slc singularities. 
\end{theorem}

\begin{remark}
 $\widehat \calM$ is the blow-up of the most singular point of $\overline{\calM}$ in the sense that $\omega\in \overline\calM$ is the only point with not almost abelian stabilizer. It follows that $\widehat{\calM}$ has only toric singularities. Kirwan--Lee \cite{kirwanlee} have constructed a full partial desingularization of $\overline{\calM}$ (i.e.  blown-up $\widehat \calM$ along the strata with toric stabilizers). 
 While this full desingularization is essential for cohomological computations on the moduli space, these extra blow-ups do not seem relevant here. 
\end{remark}

\begin{remark} We note that the locus of unigonal $K3$ surfaces gives a divisor in $\calF_2$. In fact, at the level of period domains $\calD/\Gamma_2$, the unigonal $K3$ surfaces correspond to an irreducible Heegner divisor $\calH_\infty/\Gamma_2$, where $\calH_{\infty}$ is the hyperplane arrangement associated to the rank $2$ lattice $\left(\begin{matrix} 0&1\\1&-2\end{matrix}\right)$. Theorem \ref{thmshah1} (combined with Mayer's result) gives the  isomorphism $\calM\cong (\calD\setminus\calH_{\infty})/\Gamma_2$. 
Then, Theorem \ref{thmshah2} identifies the unigonal divisor with (an open subset of) the exceptional divisor of  $\widehat \calM \to \overline \calM$.
\end{remark}

\begin{figure}[htb!]
\includegraphics[scale=0.57]{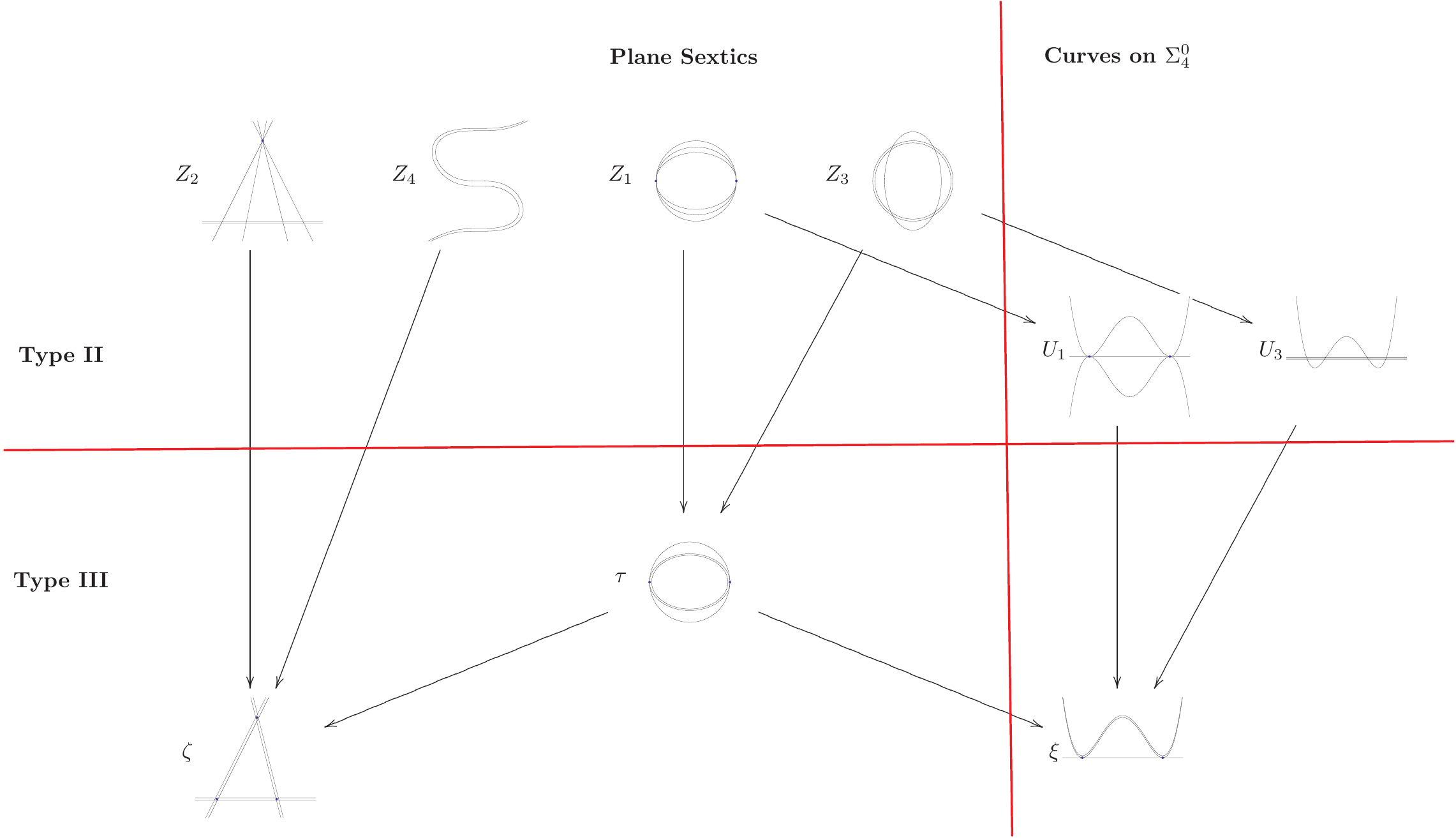} 
\caption{The minimal orbits parameterized by $\partial \widehat \calM$}\label{fig3}
\end{figure}

As stated above, the boundary components of $\calF_2\subset \widehat\calM$ are the strict transform $\widehat Z_i$ of the strata $\overline{Z}_i\subset \overline\calM$  (i.e. closures of $Z_i$). Clearly, $\overline{Z}_2$ and $\overline{Z}_4$ are unaffected by the blow-up of $\omega$. On the other hand, $\widehat Z_i\to \overline{Z}_i$ for $i=1,3$ are blow-ups of the point $\omega$ on the surfaces $\overline{Z}_i$. This introduces the exceptional divisors $\widehat U_i\subset \widehat{Z}_i$ (with open stratum $U_i$). The two exceptional divisors intersect the strict transform $\widehat \tau$ of $\overline\tau$ in a point $\xi$.  We have the following correspondence with the strata of Shah (see also Thm. \ref{thmunigonal}):
\begin{itemize}
\item[i)] $U_1$ corresponds to \cite[Thm. 4.3 Case 1(ii)]{shah}, the minimal orbits parameterize $3$ rational normal curves of degree $4$ (hyperplane sections of $\Sigma_4^0$) tangent in 2 points, giving two $\Ee$ singularities;
\item[ii)] $U_3$ corresponds to \cite[Thm. 4.3 Case 2(i)]{shah}, the minimal orbits parameterize 2 rational normal curves of degree $4$ meeting transversely, one of them counted with multiplicity $2$. This case is in fact stable. 
\item[iii)] $\xi$ corresponds to \cite[Thm. 4.3 Case 2(ii)]{shah}, the minimal orbit parameterizes 2 rational normal curves tangent in $2$ points, and one of them counted with multiplicity $2$. 
\end{itemize}
The geometry of the minimal orbits corresponding to the boundary of $\widehat \calM$ is schematically summarized in  Figure \ref{fig3} (taken from \cite{looijengavancouver}).

%%% Looijenga's work
\subsection{Comparison of the GIT and Baily-Borel compactifications}\label{sectlooijenga} As discussed above, there are two natural compactifications for the moduli space of degree $2$ $K3$ surfaces:  $\calF_2\subset \widehat \calM$ (the Shah/Kirwan GIT construction) and $\calF_2\subset (\calD/\Gamma_2)^*$ (the Baily-Borel compactification). Since the singularities of the surfaces corresponding to the boundary of $\widehat \calM$ are slc (or ``insignificant cohomological singularities''), Shah \cite{shahinsignificant,shah} noted that  there is a well-defined extended period map $\widehat \calM\to (\calD/\Gamma_2)^*$. 
A little later, Looijenga  \cite{looijengavancouver, looijengacompact} gave a precise relationship between the two compactifications as summarized below.

\begin{figure}[htb!]
\includegraphics[scale=0.57]{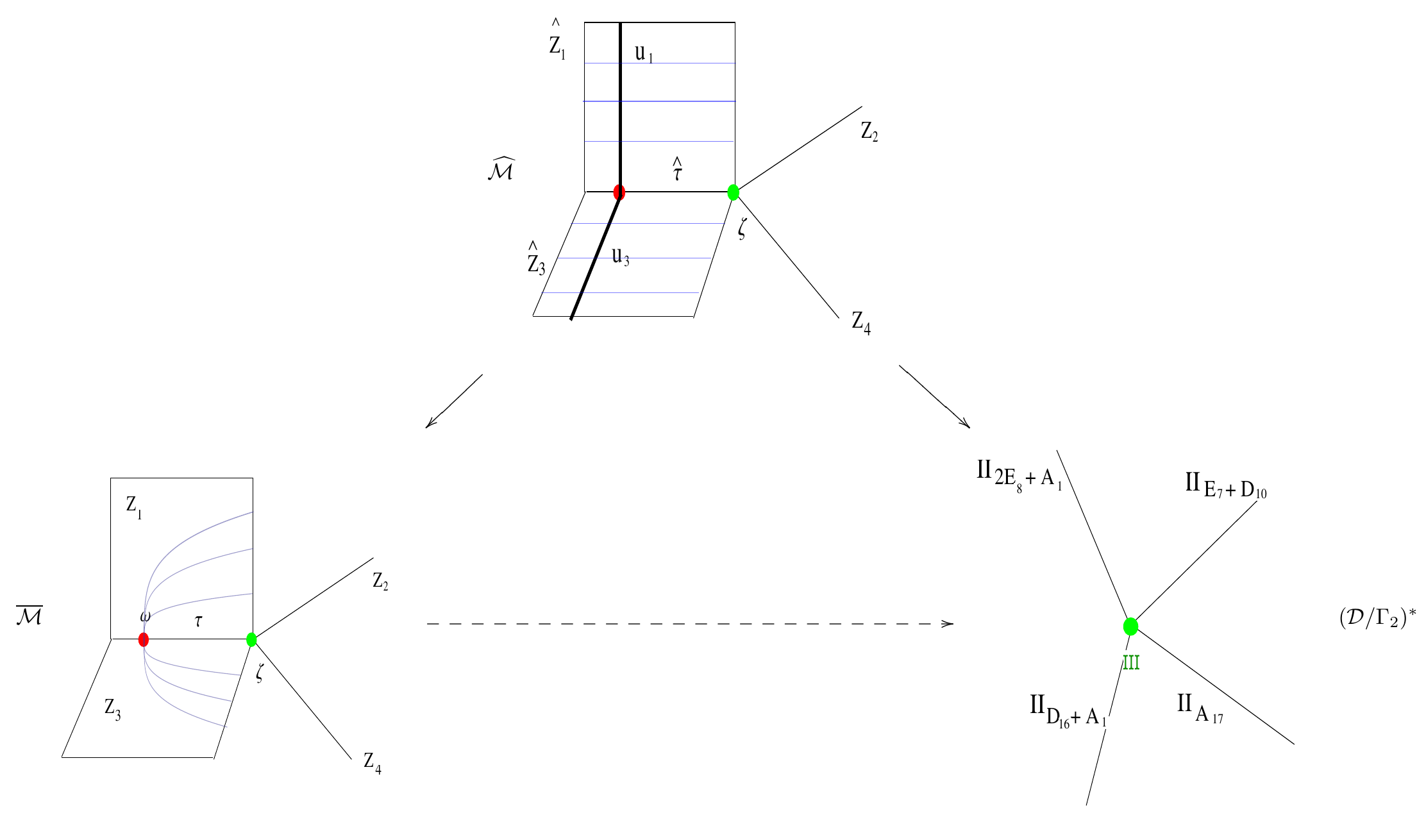} 
\caption{The boundary strata of $\widehat{\calM}$}\label{fig4}
\end{figure}

\begin{theorem}[Looijenga] \label{thmlooijenga}
The open embeddings $\calF_2\subset \widehat{\calM}$ and $\calF_2\subset (\calD/\Gamma_2)^*$ extend to a diagram (with regular maps):
$$
\xymatrix{
&\widehat{\calM}\ar@{->}[ld]\ar@{->}[rd]&\\ 
\overline{\calM}\ar@{-->}[rr]&&(\calD/\Gamma_2)^*&
}
$$
such that
\begin{itemize}
\item[i)] $\widehat \calM\to \overline \calM$ is the partial Kirwan blow-up of $\omega\in \overline\calM$;
\item[ii)] $\widehat \calM\to (\calD/\Gamma_2)^*$ is the Looijenga modification of the Baily-Borel compactification associated to the hyperplane arrangement $\calH_\infty$ (see \cite{looijengacompact}); more intrinsically, it is a small modification of $(\calD/\Gamma_2)^*$ such that the closure of Heegner divisor $\calH_\infty/\Gamma_2$ becomes $\bQ$-Cartier. 
\item[iii)] The exceptional divisor of $\widehat \calM\to \overline \calM$ maps to the unigonal divisor. 
\item[iv)] The boundary components are mapped as in Figure \ref{fig4}. 
\end{itemize}
\end{theorem}

\begin{remark}\label{remjinv}
Shah's results (\cite{shah}) give a set-theoretic extension of the period map from $\widehat \calM$ to $(\calD/\Gamma_2)^*$ without matching the strata. Scattone \cite{scattone} has computed the Baily-Borel boundary strata. The first matching of the strata (without any extension claim for the period map) is Friedman \cite[Rem. 5.6]{friedmanannals}. Finally, Looijenga's results (\cite{looijengavancouver,looijengacompact}) give that the map $\widehat \calM\to (\calD/\Gamma_2)^*$ is analytic (and thus algebraic). Additionally, it follows that:
\begin{itemize}
\item[i)] For $i=2,4$: $Z_i\cong \mathfrak H/\SL(2,\bZ)\cong \bA^1$, the map being given by the $j$-invariant associated to the minimal orbits $x_0^2f_4(x_1,x_2)$ for $Z_2$ and $f_3(x_0,x_1,x_2)^2$ for $Z_4$. The $Z_4$ case is stable, thus the orbits are in one-to-one correspondence with the points of $Z_4$, but for $Z_2$ many orbits degenerate to the same minimal orbit. Even in this case, the $j$-invariant is well defined. Namely, $Z_2$ parameterizes two different cases:  a sextic containing a double line meeting the residual quartic in $4$ distinct points, or a curve with $\widetilde{E}_7$ singularities; there is an obvious $j$-invariant in both cases.
\item[ii)] For $i=1,3$, there are rational maps $\overline{Z}_i\to \bP^1$, which are given by $j$-invariants; they are undefined at $\omega$. After the blow-up of $\omega$, we get regular maps $\widehat Z_i\to \bP^1$ (essentially $\bP^1$-fibrations).  The fibers correspond to configurations of conics such that the $j$-invariant is unchanged. For example, for $Z_3$, fix the double conic and $4$ points on it (this fixes the $j$-invariant), then the fiber of $\widehat{Z}_3\to \bP^1$  is the pencil of conics passing through these $4$ points. 
\end{itemize}
\end{remark}

%%%%%%%%%%%%%%%%%%%%%%%%%%%%%%%%%%%%%%
%%% Part 1 - The KSBA Compactification (General Discussion)
%%%%%%%%%%%%%%%%%%%%%%%%%%%%%%%%%%%%%%
%%% Sect2 : The KSBA compactification
\section{The KSBA compactification for log $K3$ surfaces}\label{sectksba}
As discussed above, the Shah--Looijenga compactification $\widehat \calM$ for $\calF_2$ has several good properties including  that the boundary points correspond to Gorenstein surface with slc singularities (higher dimensional analogues of the nodal curves). However, from a moduli 
point of view, a serious  problem with $\widehat \calM$ is that it is not separated at the boundary. In this section, we address this issue by applying  the general approach of Koll\'ar--Shepherd-Barron \cite{ksb} and Alexeev \cite{alexeevpairs0}  ({\it KSBA})   to compactifying moduli spaces of varieties of (log) general type.

In order to apply the KSBA compactifying approach, we need to change the moduli problem from $K3$ surfaces to varieties of log general type. A natural solution (e.g. \cite[\S5.1]{alexeevpairs0}) is to consider instead of $\calF_d$ {\it the moduli stack $\calP_d$} of pairs $(X,H)$ consisting of $K3$ surfaces together with an ample divisor $H$ of degree $d$; we call such pairs {\it degree $d$ $K3$ pairs}. The two moduli functors are related by the  natural forgetful map 
\begin{eqnarray*}
\calP_d&\to& \calF_d\\
(X,H)&\to& (X,\calO_X(H)),
\end{eqnarray*} 
which realizes $\calP_d$ as  a $\bP^g$-fibration (with $d=2g-2$) over $\calF_d$. 
\begin{proposition}
With notation as above, both $\calF_d$ and $\calP_d$ are smooth Deligne--Mumford stacks. Furthermore, the forgetful map  $\calP_d\to \calF_d$ is smooth and proper with fibers isomorphic to $\bP^g$. 
\end{proposition}
\begin{proof}
The smoothness of the moduli functor $\calF_d$ is well known. For a big and nef divisor $H$ on a $K3$ surface, $h^i(\calO_X(H))=0$ for $i>0$, and then
$$h^0(\calO_X(H))=2+\frac{H^2}{2}=p_a(H)+1=g+1.$$ 
The smoothness of forgetful map $\calP_d\to\calF_d$ follows from the fact that $H^1(\calO_X(H))=0$ which gives that every section of $L:=\calO_X(H)$ extends to a first-order deformation $(\sX,\calL)$ of $(X,L)$ 
 (see \cite[Prop. 3.3.14]{sernesi}; see also \cite[\S5]{beauvillek3}). Finally, since the automorphism group (as a polarized variety) of a polarized $K3$ surface is finite, it follows that $\calF_d$ and $\calP_d$ are Deligne-Mumford stacks.
\end{proof}

\begin{remark}\label{remdef}
We note that $H^1(X,L)=0$ for all degenerations of $K3$ surfaces considered in this paper (see Def. \ref{defstable} below). Thus, the forgetful map $\mathrm{Def}(X,H)\to \mathrm{Def}(X,L)$ is always smooth in our situation (here $H$ is an ample Cartier divisor and $L=\calO_X(H)$ is the associated invertible sheaf). Specifically, Kodaira vanishing ($H^i(X,L^{-1})=0$ for $L$ ample and $i=0,1$) holds if  $X$ is a demi-normal  (i.e. $X$ satisfies $S_2$ and is normal crossing in codimension $1$) projective surface (see \cite[Thm. 3.1]{arapura}, and also \cite{kovacs}). By definition, an slc variety is demi-normal. In our situation, we are also assuming $X$ is Gorenstein with $\omega_X\cong \calO_X$. Thus, by duality, we get $H^i(X,L)=0$ for $i=1,2$. By flatness, we also get $h^0(X,L)=g+1$.
\end{remark}

 Since the pairs $(X,H)$ are of log general type,  the KSBA theory gives a natural compactification for $\calP_d$ by allowing degenerations that satisfy a condition on the singularities of the pair (i.e. slc singularities) and a ``stability condition'' (i.e. ampleness for the polarization). More precisely, one has some flexibility in the definition of the moduli points by allowing a coefficient for the polarizing divisor $H$ (see  \cite{hassettweighted} for a similar situation in dimension $1$); only some choices for the coefficients give compactifications for $\calP_d$ (see Remark \ref{remcoef}). In our situation, we want the KSBA  compactification of $\calP_d$ to  be closely related to some compactifications of $\calF_d$. Thus, we would like that the choice of a divisor in a linear system to be mostly irrelevant. This is achieved by working with the moduli of pairs with $0<\epsilon\ll 1$ coefficients as in  Hacking \cite{hacking}. By adapting the general KSBA framework to our situation (see Remark \ref{remksbadef}), we define the limit objects in a compactified moduli stack $\overline{\calP}_d$ to be {\it stable pairs} as follows:
  \begin{definition}\label{defstable}
Let $X$ be a surface, $H$ an effective divisor on $X$ and $d=2g-2$ an even positive integer. We say that the pair $(X,H)$ is {\bf a stable $K3$ pair of degree $d$} if the following conditions are satisfied:
\begin{itemize}
\item[(0)] $X$ is Gorenstein  with $\omega_X\cong \calO_X$.
\item[(1)] The pair $(X,\epsilon H)$ is semi log canonical for all small $\epsilon >0$.
\item[(2)] $H$ is an ample Cartier divisor.
\item[(3)] There exists a flat deformation $(\sX,\calH)/T$ of $(X,H)$ over the germ of a smooth curve such that the general fiber $(X_t,H_t)$  is a degree $d$ $K3$ pair. Additionally, it is assumed that  $\calH$ is a relative effective Cartier divisor. 
\end{itemize}
\end{definition}

\begin{remark} Clearly, $(X,\epsilon H)$ is slc implies that $X$ is slc. 
Conversely, if $X$ is slc,  $(X,\epsilon H)$ is slc for small $\epsilon$ is equivalent to saying that $H$ does not pass through a log canonical center. In our situation ($X$ slc and Gorenstein),  this means that $H$ does not contain a component of the double locus of $X$ and does not pass through  a simple elliptic or cusp (possibly degenerate) singularity (see \cite[Thm. 4.21]{ksb}).  By working with $\epsilon$ coefficients, the singularities of the divisor $H$ are irrelevant. 
\end{remark}

\begin{remark}\label{remksbadef} The previous definition is standard  in a log general type situation with the exception of the  requirements that $X$ is Gorenstein and $H$ Cartier. In fact, the standard requirement in the KSBA approach is that $K_{X}+\epsilon H$ is 
$\bQ$-Cartier (e.g. \cite[Def. 6.1]{hackingsur}). If this condition holds for $\epsilon$    in an interval, then both $K_{X}$ and $H$ have to be $\bQ$-Cartier (and thus $X$ is $\bQ$-Gorenstein).  Moreover, for degenerations of $K3$ surfaces, it follows easily that $X$ has to be Gorenstein with $K_{X}$ trivial (e.g.  \cite[Lem. 2.7]{hacking} or Shepherd-Barron's Theorem \ref{thmsb}). Finally, the Cartier condition is justified by the observation that in the set-up of Theorem \ref{thmsb} it follows that $\overline{\calL}=\rho_*\calL$ is a relatively ample Cartier divisor on $\overline \sX$. In other words, for $K3$ surfaces we can assume the stronger conditions of Gorenstein (vs. $\bQ$-Gorenstein) in (0) and Cartier (vs. $\bQ$-Cartier) in (2) and still get a proper moduli space. 
\end{remark}

\begin{remark}\label{remcoef} Theorem \ref{thmsb} (Shepherd-Barron) bounds the type for the polarized surface $(X,L)$ that underlie a stable pair in the sense of Definition \ref{defstable} (see also Thm. \ref{thmthompson} for degree $2$). Since $H$ varies in a linear system ($L=\calO_X(H)$), we get also a bounded type for $(X,H)$.  
We conclude that there exists an $\epsilon_0>0$ (depending on $d$) such that for any $\epsilon\in(0,\epsilon_0)$ the stability condition (1)  does not change.
On the other hand, one can use other coefficients to compactify the moduli of pairs, say require $(X,\alpha H)$ to be slc for some fixed coefficient $\alpha\in(0,1)$. If $\alpha<\epsilon_0$, we get the stable pairs of Def. \ref{defstable}. For larger $\alpha$, typically the KSBA approach would modify even the interior of $\calP_d$. For instance, Example \ref{examplealpha} below shows that for $d=2$ there exists a $K3$ surface $X$ (with ADE singularities) and an ample Cartier divisor $H$ such that  $(X,\alpha H)$ is not log canonical for all $\alpha>\frac{1}{8}$  (in particular, $\epsilon_0\le \frac{1}{8}$ for $d=2$). It would be interesting to determine the critical values of $\alpha$ (or even $\epsilon_0$) for which the moduli problem changes. For $d= 2$, the GIT approach used in this paper can identify  some of the critical values for $\alpha$ (see also \cite{raduthesis}). For some related discussion (for del Pezzo surfaces) from the perspective of MMP see \cite{cheltsov}.
\end{remark}
\begin{example}\label{examplealpha}
Consider the following special plane sextic: $C=L+Q$, where $L$ is a line, $Q$ is a quintic with an ordinary node at $p$, and $L$ meets $Q$ with multiplicity $5$ at $p$. Then, the associated double cover $\overline{X}\to \bP^2$ will have a $D_{10}$ singularity over $p$. Let $X\to \overline{X}$ be the minimal resolution, and $\pi:X\to \bP^2$ the composite map. Let $E_i$ be the exceptional $(-2)$-curves (giving a $D_{10}$ graph), $L'$ be the strict transform of $L$ on $X$, and $H=\pi^*L$. Note that $L'$ is also a $(-2)$-curve and meets only $E_{10}$ (giving a $T_{2,3,8}$ graph; N.B. $D_{10}=T_{2,2,8}$). A simple computation shows that 
$$H=\pi^* L=\sum_{k=1}^8 k E_k+4E_9+5E_{10}+2L'.$$
 We conclude that the degree $2$ $K3$ pair $(X,\alpha H)$ (or equivalently $(\overline X,\alpha \overline H)$) is log canonical iff $\alpha\le \frac{1}{8}$.  
\end{example}

As already mentioned, a key result that allows us to conclude that the stable pairs give a compactification for $\calP_d$ is the following theorem of Shephed-Barron \cite{sbpolarization} (see also \cite{ksb} and \cite{kawamata}). 
\begin{theorem}[{Shepherd-Barron \cite[Thm. 2]{sbpolarization}}]\label{thmsb}
Let $\pi:\sX\to \Delta$ be a semistable degeneration of $K3$ surfaces with $K_{\sX}\equiv 0$ (i.e. a Kulikov degeneration). Assume $\calL\in\Pic(\sX)$ is nef and $\calL_{\mid X_t}$ is a polarization for all $t\in \Delta^*$. Then, for all $n\ge 4$, $\calL^{n}$ is generated by $\pi_*\calL^n$ and defines a birational morphism
$$\rho:\sX\to\overline{\sX}=\Proj_\Delta \left( \oplus_n\pi_*\calL^n\right)$$
defined over $\Delta$ such that 
\begin{itemize}
\item[a)] $\overline{\sX}$ is Gorenstein with $K_{\overline{\sX}}\equiv 0$;
\item[b)] $\overline{\sX}$ has canonical singularities;
\item[c)] $\overline{X}_0$ is Gorenstein with slc singularities. 
\end{itemize}
\end{theorem}

\begin{remark} We note that the original statement of \cite[Thm. 2 (i)]{sbpolarization} refers to {\it admissible singularities} for $\overline{X}_0$, but those are  precisely the Gorenstein slc singularities (see \cite[Def. on p. 34]{sbnef} and \cite[Thm. 4.21]{ksb}). Also, the condition for $\overline{\sX}$  of being Gorenstein with canonical singularities (or equivalently Gorenstein with rational singularities) is not stated in \cite[Thm. 2 (i)]{sbpolarization}, but occurs elsewhere in the text (e.g.  \cite[Thm. 2W]{sbpolarization}). Finally, under the set-up of the theorem, the items b) and c) are equivalent (see \cite[Thm. 5.1]{ksb}).
\end{remark}

Finally, we fit the stable pairs in a moduli functor $\overline{\calP}_d$ as follows:
\begin{definition}\label{defstable2}
For given scheme $B$, $\overline{\calP}_d(B)$ is the set of isomorphism classes of families $(\sX,\calH)/B$ such that
\begin{itemize}
\item[i)] $\sX/B$ is flat and proper;  
\item[ii)] $\sX$ is Gorenstein and $\omega_{\sX/B}$ is trivial;
\item[ii)]  $\calH$ is a relative effective Cartier divisor (in particular, flat over $B$);
\item[iii)] every geometric fiber $(X_t,H_t)$ is a stable pair.
\end{itemize}
\end{definition}

With these preliminaries, we can state the main result of the section: {\it $\overline \calP_d$ is a separated and proper Deligne-Mumford stack}. As usually this follows from the numerical criterion of properness/separateness. This is the content of the following theorem, whose proof is similar to that of Hacking  \cite[Thm. 2.12]{hacking} (see also the semi-stable MMP results, esp. \cite[Thm. 7.62]{km}).  

\begin{theorem}\label{thmksba}
Let $(\calX^*,\calH^*)/\Delta^*$ be a flat family of degree $d$ $K3$ surfaces over the punctured disk. Then there exists a finite surjective base change $\Delta'\to \Delta$ and a family $(\calX,\calH)/\Delta'$ of stable pairs extending the pullback to $\Delta'$ of the original family $(\calX^*,\calH^*)$ such that $\calX$ is Gorenstein with trivial $K_\calX$ and $\calH$ is an effective relative Cartier divisor. Furthermore, the family $(\calX,\calH)$ is unique up to further base change.
\end{theorem}
\begin{proof}
Start with a one-parameter family of polarized $K3$ surface $(\sX^*,\calL^*)/\Delta^*$. After a finite base change, one can assume a filling to a semi-stable family $\sX/\Delta$ with $K_{\sX}\equiv \calO_\sX$ (cf. Kulikov--Persson--Pinkham Theorem; alternatively this is a relative minimal model).  
Assume $\calL^*$ is induced by a flat divisor $\calH^*$. 
By Shepherd-Barron \cite[Thm. 1]{sbpolarization}, we can assume that 
the polarizing divisor extends to an effective relative Cartier divisor $\calH$, which then can be assumed also to be nef. Let $\calL=\calO_{\sX}(\calH)$.   By applying another result of Shepherd-Barron (Theorem \ref{thmsb} above), we obtain $(\overline\sX,\overline\calL)$ (the relative log canonical model) which satisfies all the required conditions, except possibly the condition that the limit $(\overline{X}_0,\epsilon \overline{H}_0)$  be slc.

Note that $(\overline\sX,\overline\calL)$ depends only on $\calL$, and not on the choice of divisor $\calH$. As explained elsewhere in the paper, the divisor $\calH$ is important for the separateness of $\overline\calP_d$. Namely, the sections $\calH$ of $\calL$ that give central fibers $(\overline{X}_0,\overline{H}_0)$ that are not stable pairs are not allowed in $\overline \calP_d$ and they are replaced by a different semi-stable model.  In our set-up (i.e. for $0<\epsilon\ll 1$ coefficients), the only obstruction to obtaining stable pairs  is that $ \overline{H}_0$ might pass through a log canonical center of $\overline{X}_0$. This degeneracy condition is  equivalent to $\calH$  containing a double curve of $X_0$ or passing through a triple point, or equivalently to saying $(\sX,\epsilon \calH)/\Delta$ (or the pair $(\sX,X_0+\epsilon \calH)$) is not dlt (compare \cite[Thm. 7.10]{km}). If this is the case, we apply a base change, refine the semi-stable model (see Rem. \ref{rembasechange} below), pull-back the polarizing divisor to the new model, and then apply again the Shepherd-Barron's Theorems. After appropriate base changes, $\calH$ will not contain any of the double curves or triple points of $X_0$, and we will be in  a log canonical situation (i.e. $(\sX,X_0+\epsilon \calH)$ is dlt). The claim (that any $1$-parameter family of degree $d$ $K3$ surfaces has as a limit a stable pair) follows.

Finally, two different semi-stable degenerations $(\sX_i,\calH_i)/\Delta$ that agree over $\Delta^*$ are (after a base change) dominated by a third. It follows that the limiting KSBA pairs $(\overline{X}^{(i)}_0,\overline{H}^{(i)}_0)$ are isomorphic  
 due to the standard fact of  MMP  that the log  canonical models are unique  (see  \cite[Thm. 2.12]{hacking} and \cite[\S37, Prop. 6]{kollar}).
\end{proof}

\begin{remark}\label{rembasechange}
For $K3$ surfaces it is easy to understand the effect of a base change on the central fiber $X_0$ of a semi-stable degeneration $\sX/\Delta$. Namely, for Type II degenerations, (for simplicity we can assume) $X_0=V_1\cup_E V_2$ where $V_i$ are rational surfaces glued along an anticanonical smooth (elliptic) curve in both. A base change $\Delta'\to\Delta$ of order $k\ge 2$ has the effect of introducing a curve of $A_{k-1}$ singularities along $E$. The blow-up of the total space along $E$ will give the new semi-stable model $X_0'$, which is a chain of surfaces with two rational ends ($V_1$ and $V_2$) and with some elliptic ruled surfaces in the middle. Similarly, in the Type III Case, $X_0$ is a union of rational surfaces such that the dual graph is a triangulation of $S^2$. The base change has the effect of subdividing each edge of the triangulation in $k$ parts and each triangle in $k^2$ parts in the obvious way (see \cite[p. 278]{fmbook}). We note that if the divisor $\calH$ contains a double curve or triple point, then the pull-back divisor $\calH'$ (after base change) will contain some component $V_i$ of the new central fiber $X_0'$.  Thus, in order to obtain a flat divisor over $\Delta'$, one needs to tensor $\calH'$ by $\calO(-V_i)$ as in \cite[Thm. 1(i)]{sbpolarization}; we call such an operation {\it twist}. After $\calH'$ is arranged to be flat (and not contain a double curve or triple point), the usual semi-stable MMP (\cite[Ch. 7]{km}) implies the theorem. The point that we want to emphasize is that it is the twist operation that allows to change the limit surface from $\overline X_0$ to $\overline X'_0$ (see Section \ref{secttype2} for some concrete examples). 
\end{remark}

\begin{corollary}\label{pdthm}
The moduli stack $\overline \calP_d$ of stable degree $d$ pairs is a proper and separated Deligne--Mumford stack. The associated coarse moduli space is a compactification (proper algebraic space) of the moduli space of degree $d$ $K3$ pairs. 
\end{corollary}
\begin{proof}
The properness of $\overline \calP_d$ was established by Theorem \ref{thmksba}. The fact that  $\overline \calP_d$ is  a Deligne-Mumford stack is the usual statement that  pairs of log general type have finite automorphisms (e.g. \cite[p. 328]{ksb}). Finally, the existence of a coarse moduli space follows from \cite{keelmori}. 
\end{proof}

\begin{remark}
In general, we do not expect that $\overline \calP_d$ is a smooth stack: as noted in Remark \ref{remdef} the local structure of $\overline \calP_d$ near $(X,H)$ is controlled by the deformations of $X$ as a polarized variety. It is likely that projectivity results for $\overline \calP_d$ can be obtained by applying the techniques of Koll\'ar \cite{kollarproj}. Alternatively, more in the spirit of this paper, the quasi-projectivity of $\calP_d$ might follow from GIT and the techniques of Viehweg \cite{viehweg} (see also \cite{viehweg2}). 
\end{remark}

\begin{remark}
The results of Shepherd-Barron  cited above established the existence of reasonable limits for degenerations of polarized $K3$ surfaces. In some sense, the Shah--Looijenga compactification $\widehat \calM$ is a reflection of this fact. However, in absence of a polarizing divisor, the limiting surfaces will not be separated in moduli. For example, the limiting surfaces might not have finite stabilizer, e.g.  the standard tetrahedron in $\bP^3$ is stabilized by a torus, leading to  collapsing of orbits. The presence of a divisor giving a log canonical log general type pair eliminates such pathologies. 
 In other words, the choice of a divisor (vs. line bundle) is essential in separating the boundary points and fitting everything together in a compact moduli space. The example discussed in \S\ref{sectexample} is a  clear illustration of this point. 
\end{remark}

For degree two $K3$ surfaces, the possible central fibers $\overline{X}_0$ of  relative log canonical models  were identified by Thompson \cite{thompson}. Some discussion of $\overline{X}_0$ for general $d$ is done in the following section.

\begin{theorem}[{\cite[Thm. 1.1]{thompson}}]\label{thmthompson}
Let $\sX/\Delta$ be a Kulikov degeneration of $K3$ surfaces. Let $\calH$ be a divisor on $\sX$ that is effective, nef and flat over $\Delta$. Suppose that $\calH$ induces a polarization of degree two on the generic fiber $X_t$. Then the morphism $\phi:\sX\to \overline{\sX}$ taking $\sX$ to the relative log canonical model  of the pair $(\sX,\calH)$ maps the central fiber $X_0$ to  a complete intersection of the following type:
$$\overline{X}_0=\{z^2-f_6(x_i,y)=f_2(x_i,y)=0\}\subset \bP(1,1,1,2,3).$$
\end{theorem}

\begin{remark}
Thompson \cite{thompson} does not consider the polarizing divisor as part of the data so that it is not possible to fit the degenerations in a  moduli space. In fact, no  attempt of constructing a moduli space is made in \cite{thompson}. As explained, keeping track of the polarizing divisor allows us to construct a modular compactification for pairs. Also, even if one is only interested in $K3$ surface, by considering pairs one has a better understanding of how the various points of view - GIT (\cite{shah}), Hodge theoretic (\cite{friedmanannals}, \cite{friedmanscattone}), or abstract MMP (\cite{thompson}) - interact (see Sections \ref{secttype2} and \ref{secttype3} for some concrete examples). 
\end{remark}

%%%%%%%%%%%%%%%%%%%%%%%%%%%%%%%%%%%%%%
%%% Sect3 :  Polarized Anticanonical pairs
\section{Classification of polarized anticanonical pairs in degree two}\label{sectpolanti}
In order to understand the possible boundary points of $\overline \calP_d$, we need to understand the possible central fibers $\overline{X}_0$ of relative log canonical models as in the previous section.  We recall that $\overline{X}_0$ is a contraction of the central fiber $X_0$ of a Kulikov model. Then, depending on the index of nilpotency of the monodromy, the normal crossing variety $X_0=\cup V_i$ is (see \cite[p. 11]{fmbook})
\begin{itemize}
\item either of {\it Type II}, i.e. a chain of surfaces glued along elliptic curves with rational ends and elliptic ruled surfaces in the middle,
\item or of {\it Type III}, rational surfaces such that the dual graph gives a triangulation of $S^2$, the double curves on each $V_i$ form a cycle of rational curves, which is an anticanonical divisor on $V_i$. 
\end{itemize}
Note also that $\overline{X}_0$ depends only on the polarized semi-stable model $(X_0,L_0)$ and not on the degenerating family $(\sX,\calL)$ (see \cite[Lem. 2.17]{sbpolarization}, \cite[Lem. 4.1]{thompson}). In fact, the analysis of Shepherd-Barron \cite{sbpolarization} says that  $\overline{X}_0$ can be essentially recovered from the $0$-surfaces $(V_i,L_i)$ in $X_0$ (with $L_i=L_{0\mid V_i})$, i.e. the components of $X_0$ that are mapped birationally  onto  the image (see  \cite[Def. on p. 145]{sbpolarization}).

Thus, to understand the boundary points in $\overline \calP_d$, it is essential to classify the possible $0$-surfaces that can occur in degree $d$. Note that on a $0$-surface $V_i$ the polarization $L_i$ is big and nef. Also, the degrees of the polarizations on all $0$-surfaces of a polarized semistable $X_0$ satisfy $\sum (L_i)^2=d$. 
Thus, we need to classify  triples $(V,D;L)$, where $(V,D)$ is an anticanonical pair and $L$ is big and nef divisor class with $1\le L^2\le d$. 
 To fix the notation and terminology, we define the following:

\begin{definition}
A {\bf polarized anticanonical surface} is a triple $(V,D;L)$ where
\begin{itemize}
\item[i)] $V$ is a rational surface,
\item[ii)] $D\in|-K_X|$ is a reduced anticanonical divisor, 
\item[iii)]  $L\in \Pic(V)$ is a big and nef divisor class.
\end{itemize}
We say  $(V,D;L)$ is  {\bf relatively minimal} if  any minus-one curve $E$ on $V$ satisfies $L.E>0$. Additionally, we will be mostly concerned with the case that $D$ is at worst nodal, in which case we say $(V,D)$ is of {\bf Type II} or {\bf Type III} if $D$ is a smooth (elliptic) curve or $D$ is a cycle of rational curves  respectively. \end{definition}
\begin{remark}
Any anticanonical pairs $(V,D)$ can obtained by a series of blow-ups of a minimal anticanonical pair (a classification  of such is \cite[Lemma 3.2]{fmiranda}). Specifically, given an anticanonical pair $(V',D')$, the blow-up of a point $p\in D'$ gives another anticanonical pair $(V,D)\to (V',D')$, where $D=\pi^*D-E$.  If $p$ is a node of $D'$ we say that such a blow-up is {\it toric}; if $p$ is smooth on $D'$, we call it {\it non-toric}.  
Consider a blow-up $\pi: (V,D)\to (V',D')$ of anticanonical pairs. Let $L'$ be a big and nef divisor on $(V',D')$ and $L=\pi^*L'$. Clearly, $L$ is still big and nef and the following hold: $L^2=(L')^2$, $L.D=L'.D'$, and $D^2=(D')^2-1$. 
\end{remark}

The previous remark makes clear that for  a meaningful classification of the polarized anticanonical surfaces $(V,D;L)$ it is necessary to assume them to be relatively minimal. 
We note that the relatively minimal condition is a purely numerical condition. Thus,  if needed, we can assume that $(V,D;L)$ is relatively minimal (see also \cite[Lem. 2.12(a)]{har2}).  More precisely, a standard application of Riemann--Roch and Hodge index shows that a class $E$ with $E^2=-1$, $E.D=1$, and $E.L=0$ is effective and contains a $(-1)$-curve (orthogonal to $L$) as component. After successive contractions of $(-1)$-curves orthogonal to the polarization, we obtain a relatively minimal surface $(V',D';L')$ such that $\pi:(V,D)\to (V',D')$ is a composition of blow-ups as in the previous remark and $L=\pi^*L'$.

%%% Numerical bounds in terms of degree
\subsection{Basic observations on polarized anticanonical surfaces} A nef divisor on an anticanonical pair is  always effective (e.g. \cite[Cor. 2.3]{har2}), and in many situations it is easy to compute the dimension of the corresponding linear system. 
\begin{proposition}[{\cite[Lemma 5]{friedmanlin}, \cite[Thm. I.1]{har1}}]\label{proprr}
Let $(V,D)$ be an anticanonical pair and $L$ be a nef divisor. The following hold:
\begin{itemize}
\item[a)] If $D.L>0$, then $h^1(L)=0$. Thus, 
$$h^0(L)=\frac{(L^2+L.D)}{2}+1.$$
\item[b)] If $D.L=0$ and $|L|$ contains a reduced connected member, $h^1(L)=1$.  Thus,
$$h^0(L)=\frac{(L^2+L.D)}{2}+2.$$
\end{itemize}
\end{proposition}

As will see in many cases it is possible to classify the polarized anticanonical surfaces $(V,D;L)$ based on the  basic  numerical invariants $L^2$, $L.D$, and $D^2$. As discussed, for degenerations of $K3$ surfaces occurring in degree $d$, we have $1<L^2\le d$.  The following lemmas establish some 
 bounds for $L.D$ and $D^2$ in terms of $L^2$. 
\begin{lemma}\label{lbound1}
Let $(V,D;L)$ be a polarized anticanonical surface. Then
\begin{itemize}
\item[i)] $L.D\equiv L^2 \mod 2$;
\item[ii)] $0\le L.D\le L^2+2$.
\end{itemize} 
\end{lemma}
\begin{proof}
The first part follows from the fact that the orthogonal complement in $\Pic(V)$ of the canonical class $K_V(=-D)$ is an even lattice. 

For $L.D\ge 3$ the linear system $L$ is base point free and defines a birational map (e.g. \cite[Prop. 3.2]{har2}). Thus, to prove ii), without loss of generality we can assume  that the general member of $L$ is reduced and irreducible. Then, $2p_a(L)-2=L^2-L.D\ge -2.$
\end{proof}

To control $D^2$, we distinguish two cases: either $L.D\le L^2$ or $L.D=L^2+2$. To handle the first case the key observation is that it is possible to {\it twist  the polarization} (compare Rem. \ref{rembasechange}), i.e. replace $L$ by $L-D$. 
\begin{lemma}\label{lemtwist}
Let $(V,D;L)$ be a polarized anticanonical surface. Then $L-D$ is effective iff $L.D\le L^2$.  
\end{lemma}
\begin{proof}
Note that $L_0:=L-D=L+K_V$ is an adjoint linear system with $L$ big and nef. Thus, by Kodaira-Mumford vanishing,  $h^i(L_0)=0$. We conclude $h^0(L_0)=1+\frac{1}{2}(L-D)L$; the claim follows.
\end{proof}
\begin{lemma}\label{lbound2}
Let $(V,D;L)$ be a polarized anticanonical surface. Assume additionally that $(V,D;L)$ is relatively minimal and $L.D\le L^2$. The following hold:
\begin{itemize}
\item[i)] $L-D$ is nef;
\item[ii)] $2L.D-L^2\le D^2\le L.D$, and the inequality on the right is strict unless $L\sim D$. 
\end{itemize}
\end{lemma}
\begin{proof}
The first part is precisely \cite[Lem. III.9(c)]{har1} (use $L-D$ is effective by Lemma \ref{lemtwist}). Since $L-D$ is nef, we get $(L-D)^2\ge 0$, which gives the first inequality above. The second inequality follows from Hodge index: $D^2\le \frac{(L.D)^2}{L^2}\left (\le L.D \right)$.  
\end{proof}
In particular, we note the following classification result:
\begin{corollary}\label{cordp}
Let $(V,D;L)$ be a relatively minimal polarized anticanonical surface. Assume that $L.D=L^2$. Then $V$ is a del Pezzo surface and $L\sim D$. 
\end{corollary}
\begin{proof}
From Lemma \ref{lbound2}, we get 
$$L^2=2L.D-L^2\le D^2\le L.D=L^2.$$
Thus, $D^2=L.D=L^2$. From Hodge index  applied to the classes $L$ and $D$, we conclude $L\sim D$. It follows that $V$ is a rational surface with a big and nef anticanonical divisor, thus a del Pezzo (possibly with ADE singularities). 
\end{proof}

It remains to consider the case $L.D=L^2+2$. Again, a classification is readily available. 

\begin{proposition}\label{propclassify}
Let $(V,D;L)$ be a polarized anticanonical surface. Assume additionally that $(V,D;L)$ is relatively minimal and that $L.D= L^2+2$. Then,
\begin{itemize}
\item[i)] either   $V\cong \bP^2$ with polarization $L=\ell$ or $2\ell$ (where $\ell$ is the class of a line),
\item[ii)] or $(V,L)$ is the rational normal scroll, i.e. $V\cong \bF_n$ and $L=\sigma+(n+k)f$ for some $k\ge 0$ (where $\sigma$ is the class of the negative section, and $f$ is the class of a fiber). 
\end{itemize}
Moreover, the pairs $(V,D)$ with $V\cong \bP^2$ or $\bF_n$ are classified by \cite[Lem. 3.2]{fmiranda}.
\end{proposition}
\begin{proof}
Let $L^2=n\ge 1$. Since $L.D=n+1\ge 3$, from \cite[Prop. 3.2]{har2} it follows that $L$ is base point free defining a birational morphism from $V$ to a normal surface $V'\subset \bP^{n+1}$ (cf. Prop. \ref{proprr}(i)) of degree $n$. It follows that $V'$ is a surface of minimal degree (see \cite[p. 525]{gh}) and thus it is either the Veronese surface or the rational normal scroll (i.e. $\bF_n$ embedded by $\sigma+(n+k)f$; the case $n=1$, $k=0$ gives $(\bP^2,\ell)$). Finally, note that the morphism $V\to V'$ contracts the curves orthogonal to $L$ and those curves are not $(-1)$-curves. The proposition follows. 
\end{proof}

\begin{remark}
Note $D^2=K_V^2\le 9$ for all rational surfaces $V$. In fact, $b_2(V)=10-D^2$, and then $D^2$ is $9$ or $8$ only for $\bP^2$ or $\bF_n$ respectively. 
For Type III anticanonical pairs, one also considers $r(D)$ the length of the anticanonical cycle, and {\it the charge} $q(V,D):=12-D^2-r(D)$ (e.g. \cite[\S3]{fmiranda}).
Roughly, $9-D^2$, $r(D)-3$, and $q(D)$ count the total number of blow-ups, the number of toric blow-ups, and the number of non-toric blow-ups respectively. If $(V,D)$ is a component of a Type III degeneration of $K3$ surfaces,  $0\le q(V,D)\le 24$ (e.g. \cite[\S3]{fmiranda}).
\end{remark}

%%% Friedman/Harbourne Results
\subsection{Linear systems on anticanonical surfaces} We now recall some results on the behavior of linear systems on anticanonical pairs analogous to Mayer's Theorem for $K3$ surfaces.  Results on this topic were first obtained by Friedman \cite{friedmanlin}, and then strengthened by Harbourne \cite{har1,har2}. The following  holds:

\begin{theorem}[{Harbourne \cite[Cor. 1.1]{har2}}]
Let $(V,D;L)$ be a polarized anticanonical surface. Then $|3L|$ always defines a birational morphism from $V$ onto the normal surface obtained by contracting all curves $C$ on $V$ orthogonal to $L$.  
\end{theorem}

Similar to $K3$'s , we have the following results on base loci of linear systems on anticanonical surfaces.  For clarity, we separate the cases $L.D>0$ and $L.D=0$.

\begin{theorem}[Friedman, Harbourne]\label{thmldn0}
Let $(V,D;L)$ be a polarized anticanonical surface. Assume that $(V,D;L)$ is relatively minimal and that $L.D>0$. The following hold:
\begin{itemize}
\item[i)] If $L.D\ge 2$, then $|L|$ is base point free. Furthermore, if $L.D\ge 3$, then $|L|$ defines a birational morphism onto a normal surface. 
\item[ii)] If $L.D=1$ and $L$ has no fixed component, then $|L|$ has a unique base point, which is on $D$.
\item[iii)] if $L.D=1$, $|L|$ has a fixed component iff 
$$L=kE+R, \textrm{ for some } k\ge 2,$$ 
where $E^2=0$, $E.D=0$, $E.R=1$, and  $R$ is a $(-1)$-curve.
\end{itemize}
\end{theorem}
\begin{proof}
The first two items follow directly from \cite[Thm. III.1 (a,b)]{har1} and \cite[Prop. 3.2]{har2} (see also \cite[Thm. 10]{friedmanlin}). The last statement follows also from \cite[Thm. III.1]{har1} after contracting the $(-1)$-curves orthogonal to $L$.
\end{proof}

\begin{theorem}[Friedman, Harbourne]\label{thmld0}
Let $(V,D;L)$ be a polarized anticanonical surface. Assume $L.D=0$. Then, one of the following holds:
\begin{itemize}
\item[i)] either $L$ has no fixed component, then $L$ is base point free, $L\otimes \calO_D$ is trivial, and $h^1(X,L)=1$;
\item[ii)] or the fixed part of $L$ is a $(-2)$-curve $R$, then 
$$L=kE+R, \textrm{ for some } k\ge 2$$
with $E^2=E.D=0$, $E.R=1$, and $R\otimes \calO_D$ trivial;
\item[iii)] or $L\otimes D$ is non-trivial, which is equivalent to saying that  the fixed part $F$ of $L$ satisfies $F+K_X$ is an effective divisor.  In this situation, there exists a birational morphism $\pi:(V,D)\to (V',D')$ of anticanonical pairs with $(D')^2<0$ and such that $L=\pi^*(L'+D')$ for some nef divisor $L'$ on $V'$. 
\end{itemize}
\end{theorem}
\begin{proof}
This is precisely \cite[Thm. III.1 (c,d)]{har1} assuming  $L$ big. 
\end{proof}

The items Thm. \ref{thmldn0}(iii) and Thm. \ref{thmld0}(ii) correspond precisely to the unigonal case of Mayer's Theorem. Also, since we are considering only slc pairs, we can assume (if necessary) that $L$ does not contain $D$ as  a fixed component.

\begin{remark}
The key fact that allows Harbourne \cite{har1,har2} to strengthen the results of Friedman \cite{friedmanlin}  is a precise control of Friedman's condition: {\it $L$ has no fixed component which is also a component of the anticanonical cycle}. Namely, \cite[Cor. III.3]{har1} says: {\it Given $L$ a nef divisor on an anticanonical pair $(V,D)$, then either no fixed component of $L$ is a component of any section of $-K_V$ or the fixed part of $L$ contains an anticanonical divisor.} This situation can only occur if $L.D=0$, but $L_{\mid D}\not\cong \calO_D$ (see Thm. \ref{thmld0}(iii) above). 
\end{remark}

%%% Classification of anticanonical pairs in degree 2
\subsection{The degree $2$ case}
We now restrict to the case $(V,D;L)$ is a $0$-surface in a degeneration of degree $2$ $K3$ surfaces. The above discussion leads to the following simple classification of the possibilities.

\begin{proposition}\label{classifydeg2}
Let $(V,D;L)$ be a relatively minimal polarized anticanonical surface with $L^2\le 2$. Then one of the following six cases holds:

\noindent (A) If $L^2=1$
\begin{itemize}
\item[(1)]  and $L.D=3$, then $V\cong \bP^2$, $L\sim \ell$ (where $\ell$ is the class of a line);
\item[(2)] and $L.D=1$, then $V$ is a degree $1$ del Pezzo and $L\sim D\sim -K_V$. 
\end{itemize}
\noindent (B) If $L^2=2$
\begin{itemize}
\item[(3)] and $L.D=4$, then $V$ is an irreducible reduced quadric in $\bP^3$ and $L$ the class of a hyperplane section; 
\item[(4)] and $L.D=2$, then $V$ is a degree $2$ del Pezzo with $L\sim D\sim -K_V$; 
\item[(5)] and $L.D=0$ and $D^2=-1$, then $(V,D)$ is the resolution of a rational surface which has a unique non-ADE singularity, which is either a simple elliptic singularity of type $\Ee$ or a a cusp singularity of type $T_{2,3,r}$ (with $7\le r\le 16$);
\item[(6)] and $L.D=0$ and $D^2=-2$, then $(V,D)$ is the resolution of a rational surface which has a unique non-ADE singularity, which is either a simple elliptic singularity of type $\Es$ or a cusp singularity of type $T_{2,q,r}$ (with $q\ge 4$, $r\ge 5$, $q+r\le 19$).
\end{itemize}
\end{proposition}
\begin{proof}
The possible values for $L.D$ and $D^2$ are determined by the Lemmas \ref{lbound1} and \ref{lbound2}. The first four items follow from Corollary \ref{cordp} and Proposition \ref{propclassify}. The statement about the type of singularities for cases (5) and (6) is standard. Finally, for the bounds on $q$ and $r$, we note that  the charge associated to a cusp lying on a rational surface is at most $21$ (cf. \cite[Lem. 4.6]{fmiranda}). For $T_{p,q,r}$ singularities the associated charge is $q(V,D)=p+q+r$. Thus, $q+r\le 19$ (or $r\le 16$).\end{proof}

\begin{remark}
A precise analysis of the cusp singularities $T_{2,q,r}$ occurring in degree $2$ can be done using \cite[\S5]{wall} and \cite[\S6]{wall} for $T_{2,3,r}$ and $T_{2,q,r}$ respectively. 
\end{remark}

In the case of Type II degenerations, one might have to consider elliptic ruled components as $0$-surfaces. Here, we note that, at least for degree $2$, the elliptic ruled components  can be viewed as degenerations of rational anticanonical surfaces. 

\begin{lemma}\label{partialsmooth}
Let $(V,D',D'';L)$ be a polarized anticanonical triple (i.e. $V$ is elliptic ruled and $D'$, $D''$ are sections with $D'+D''\in|-K_V|$). Assume $V$ is relatively minimal and $L^2\in\{1,2\}$. Furthermore, assume $L$ has no fixed common component with $D'\cup D''$. Then one of the components, say $D''$, satisfies $-L^2\le (D'')^2<0$ and $D''.L=0$ and thus it can be contracted to a $\Er$  ($r\in\{7,8\}$) singularity at some point $p$ on a normal surface $\overline V$. Then, there is a partial smoothing $(\mathcal V,\calD,\calL)$ of $p$ such that the central fiber is $(\overline{V},D',L)$ and the general fiber $(V_t,D_t;L_t)$ is a polarized rational anticanonical surface. 
\end{lemma}
\begin{proof}
By \cite[Lemma 4.6]{thompson} and \cite[p. 23-24]{thompson}, we have a precise control on the surfaces $(V,D',D'';L)$  that can occur.  The claim can be checked explicitly.  For example, in the case $\Es$ non-unigonal: $V$ is a double cover of $\bP^2$ branched along the sextic $x_0^2f_4(x_1,x_2)$, a partial smoothing is given by $V(z^2-x_0^2F_4(x_0,x_1,t\cdot x_2))\to\bA^1_t$ for some polynomial $F_4$ with $F_4(x_0,x_1,0)=f_4(x_0,x_1)$. 
\end{proof}

%%%%%%%%%%%%%%%%%%%%%%%%%%%%%%%%%%%%%%
%%% Part 2 - Construction of the KSBA compactification via GIT
%%%%%%%%%%%%%%%%%%%%%%%%%%%%%%%%%%%%%%
%%% Sect4 : VGIT section
\section{A GIT construction for the moduli for pairs}\label{sectvgit}
In Section \ref{sectksba} we have shown that the moduli of degree $2$ $K3$ pairs has a geometric compactification $\overline\calP_2$. While a rough classification of the degenerate degree $2$ pairs is given by Proposition \ref{classifydeg2}, a full classification of the geometric objects parameterized by the boundary of  $\overline\calP_2$ seems  difficult to obtain by direct considerations. Instead, we study $\overline \calP_2$ by using a related GIT space $\widehat \calP_2$. 

Namely,   $\widehat \calP_2$ is constructed by enhancing the GIT analysis of Shah (giving $\widehat \calM$) to take into account a hyperplane section. This construction is closely related to that of \cite{raduthesis}. The main point here is that there is a choice of linearization involved in the construction of a 
GIT quotient for pairs, giving in fact a family of quotients $\widehat \calP_2(\alpha)$ for $\alpha\in \bQ_+$. As a limiting case, $\widehat \calP_2(0)$ is still defined and  $\widehat \calP_2(0)\cong \widehat \calM$. By general considerations from VGIT, one gets a natural forgetful map $\widehat\calP_2(\epsilon)\to \widehat \calP_2(0)\cong \widehat \calM$ (for $0<\epsilon\ll 1$), which is generically a $\bP^2$-bundle. We define $\widehat \calP_2:=\widehat\calP_2(\epsilon)$ and note that (since  $\epsilon\ll 1$) the stability conditions for  $\widehat \calP_2$ are essentially determined by Shah's stability conditions for $\widehat \calM$.  However, in $\widehat \calP_2$ more orbits are separated than in $\widehat \calM$. Finally, using Theorem \ref{thmshah2}, we conclude  that $\widehat \calP_2$ is closely related to the KSBA compactification $\overline \calP_2$. We summarize the results of the section as follows:

 \begin{theorem}\label{thmgitpair}
 The GIT quotient $\widehat \calP_2$ (constructed in this section) compactifies the moduli space of degree $2$ pairs $\calP_2$ and has the following properties: 
 \begin{itemize}
 \item[i)] $\widehat \calP_2$ has a natural forgetful map $\widehat \calP_2\to \widehat \calM$ (with generic fiber $\bP^2$);
 \item[ii)] the (GIT) stable locus $\calP_2^s\subset \widehat \calP_2$ is a moduli space of KSBA stable degree $2$ pairs $(X,H)$ such that $X$ is double cover of $\bP^2$ or $\Sigma^0_4$ (and thus $\calP_2^s$ is a common open subset of both $\widehat \calP_2$ and $\overline \calP_2$); 
 \item[iii)] the strictly semistable locus  $\widehat \calP_2\setminus \calP_2^s$ is a surface $\widetilde Z_1$ that maps one-to-one to the closure of the stratum $\widehat Z_1\subset \widehat \calM$.
 \end{itemize}
 \end{theorem}
 
 The actual construction of $\widehat \calP_2$ and the analysis of the stability conditions is the content of the section (see esp. \eqref{defnu} and \eqref{defuni} for the construction, and \ref{cornonunigonal} and \ref{corunigonal} for the analysis of stability) after the introductory example discussed in \S\ref{sectexample}.  

\subsection{A motivating example}\label{sectexample} 
We start by discussing a simple example that  illustrates how the compactification procedure described in Section \ref{sectksba} works and also hints to the relevance of GIT/VGIT to the construction of $\overline \calP_2$. Specifically, we consider  the analogous $1$-dimensional compactification problem: the moduli space of pairs consisting of an elliptic curves $E$ and a divisor $D$ of degree $d$. The definition \ref{defstable} can be easily adapted to this situation. The resulting analogue of $\overline{\calP}_d$ is precisely the moduli space of weighted stable curves $\calM_{1,\calA}$ (or more precisely $\calM_{1,\calA}/\Sigma_d$ in the notation of loc. cit.)
of  Hassett \cite{hassettweighted} for the weight system  $\calA=(\epsilon,\epsilon,\dots,\epsilon)$.  Furthermore, for small $\epsilon$, there is a natural forgetful map $\calM_{1,\calA}\to \overline{\calM}_1$, where $\overline{\calM}_1\cong \bP^1$ is the compactified $j$-line. The boundary points in $\calM_{1,\calA}$ (corresponding to the fiber over $\infty\in \overline{\calM}_1$) are easily described: they are cycles $C$ of rational curves such that each component contains at least one point of $D$ (this is the ampleness condition of Def. \ref{defstable}); the points of $D$ are allowed to coincide, but they should be distinct from the nodes of $C$ (this is the slc condition of Def. \ref{defstable}). 

When $d=3$, the moduli of pairs as above can be constructed via GIT. Namely, an elliptic curve with a degree $3$ polarization is a plane cubic $C$. If one considers instead an elliptic curve with a polarizing divisor, one gets a pair $(C,L)$ consisting of a plane cubic and a line. A GIT quotient for such pairs (i.e. plane curves plus a line) was studied in \cite{raduthesis}. Namely, we have a one-parameter VGIT situation: the GIT quotient for pairs is $\calP(\alpha)=\bP H^0(\bP^2,\calO(3))\times \check\bP^2\gquot_{\calO(1,\alpha)}\SL(3)$ for $\alpha\in \bQ_{\ge 0}$. Then, $\calP(\epsilon)\cong \calM_{1,(\epsilon,\epsilon,\epsilon)}$ (for $0<\epsilon\ll1$) and $\calP(0)\cong \overline{\calM}_1\cong \bP^1$ (the GIT quotient for plane cubics).  Furthermore, by VGIT there is a natural forgetful morphism $\calP(\epsilon)\to \calP(0)$, which coincides with $\calM_{1,(\epsilon,\epsilon,\epsilon)}\to \overline{\calM}_1$ from the previous paragraph.

There are two advantages to using the GIT construction. First, the spaces $\calP(\epsilon)\cong \calM_{1,(\epsilon,\epsilon,\epsilon)}$  and $\calP(0)\cong \overline{\calM}_1$, and the forgetful morphism $\calP(\epsilon)\to \calP(0)$ are automatically projective  (the same can be shown without GIT, but with more involved arguments). Also, the GIT description makes clear the difference between polarization and polarizing divisor. Namely,  the GIT quotient $\calP(0)\cong \bP^1$ has weak modular meaning: over $\bA^1$ the quotient is modular (each point corresponding to a unique smooth cubic), but over $\infty$ three different orbits (the nodal cubic, the conic plus a line, and the triangle) are collapsed to the minimal orbit corresponding to   the triangle in $\bP^2$ (with $(\bC^*)^2$ stabilizer). When one considers $\calP(\epsilon)$, i.e. pairs $(C,L)$ with the line given weight $\epsilon$, essentially nothing changes over the stable locus $\bA^1\subset \calP(0)$ (resulting in a $\bP^2$-fibration), but over $\infty$ the three collapsing orbits are separated. 
The point is that a nodal cubic is strictly semi-stable, but when considered together with a line it becomes either stable (if the line does not pass through the node) or unstable (if the line passes through the node). Thus, we get $\calP(\epsilon)$ is modular, in contrast to the weakly modular space $\calP(0)$. Also, it is easy to see that (up to finite stabilizer) $\calP(\epsilon)\to \calP(0)$ becomes a $\bP^2$-fibration even over $\infty$. 

\begin{remark}
Note that $\calP(0)$ parameterizes nodal cubics (analogue to the slc condition from Thm. \ref{thmshah2}), and thus the only failure of the modularity is the non-separateness at the boundary. Also, note that the limit procedure for a nodal cubic plus a line as the line approaches the node is to replace the nodal cubic by a conic plus a line (and then by a triangle); this illustrates one of the essential points of the proof of Thm. \ref{thmksba}. Finally, some general connections between GIT stability and KSBA stability was noticed by Kim--Lee \cite{kimlee} and Hacking \cite[\S10]{hacking} (essentially appropriate KSBA stability implies GIT stability). This connection is the strongest for Calabi--Yau hypersurfaces. In some sense, this is what  makes the example discussed in this section and the degree $2$ $K3$ case work.
\end{remark}

%%% GIT for pairs
\subsection{GIT for sextic pairs} The goal of the section is to construct a GIT moduli space $\widehat \calP_2$ for degree $2$ pairs together with a forgetful map $\widehat \calP_2\to \widehat \calM$.  Following Shah \cite{shah}, we do this construction in two steps. First in this subsection, we handle the non-unigonal case: we obtain an open subset $\calP_2^{nu}\subset \widehat \calP_2$ and a forgetful map $\calP_2^{nu}\to \overline \calM\setminus \{\omega\}$. Then, working near $\omega$ and invoking Luna type slice results, we obtain a neighborhood $U$ of the unigonal divisor. The gluing of $\calP_2^{nu}$ and $U$ gives $\widehat \calP_2$ together with a forgetful morphism $\widehat \calP_2\to \widehat \calM$.

The construction of $\calP_2^{nu}$ follows the example discussed in \S\ref{sectexample} (and \cite{raduthesis}). Simply, we consider the family of GIT quotients associated to pairs $(C,L)$, where $C$ is a plane sextic and $L$ is a line:
$$\calP(\alpha):=\left(\bP H^0(\bP^2,\calO(6))\times\check\bP^2\right)\gquot_\alpha \SL(3).$$
As before, we have $\calP(0)\cong \overline{\calM}$ (the GIT quotient for plane sextics described in Thm. \ref{thmshah1}) and a forgetful map  $\pi:\calP(\epsilon)\to  \overline \calM\cong \calP(0)$ (generically, a $\bP^2$-fibration) for small $\epsilon$.  We define
\begin{equation}\label{defnu}
\calP_2^{nu}=\pi^{-1}\left(\overline \calM\setminus\{\omega\}\right)\subset \calP(\epsilon),
\end{equation}
i.e. we remove from $\calP(\epsilon)$ the pairs $(C,L)$ with $C$ degenerating to the triple conic. 

\begin{notation}
In what follows, we will denote by $(C,L)$ a pair of a plane sextic and a line and by $(X,H)$ the double cover associated to it (not necessarily normal). We will use $(C,L)$ and $(X,H)$ interchangeably to refer to points of $\calP_2^{nu}$ (and  $\widehat \calP_2$). 
\end{notation}

From the general VGIT theory, it follows that if $C$ is a stable/unstable sextic, then $(C,L)$ is $\epsilon$-stable/$\epsilon$-unstable. We conclude that the pairs $(X,H)$ consisting of a $K3$ surface $X$ (possibly with ADE singularities) and an arbitrary degree $2$ ample (Cartier) divisor $H$ are stable in $\widehat \calP_2$ (N.B. strictly speaking this applies to $\calP_2^{nu}$ here, but the unigonal case is similar). Thus, {\it $\widehat \calP_2$ is a compactification of the moduli $\calP_2$ of degree $2$ $K3$ pairs}.  Also from Proposition \ref{propslc} (and then Theorem \ref{thmshah2}) the semi-stable pairs $(X,H)$ corresponding to points of $\calP_2^{nu}$ (and then similarly for $\widehat \calP_2$) have the property that $X$ is a double cover of $\bP^2$ (and later also $\Sigma_4^0$) with at worst slc singularities. What remains to be understood is how the strictly semi-stable orbits of sextics become separated when considered as pairs, and the connection between $\epsilon$-GIT stability and $\epsilon$-KSBA stability. In fact, as already mentioned in a previous remark, $\epsilon$-KSBA stability  always implies $\epsilon$-GIT stability.

We  discuss the stability of pairs based on the stratification of $\overline\calM$ given by Theorem  \ref{thmshah1}. First note that the strata $Z_4$ (double cubic) and $Z_3$ (double conic plus another transversal conic) parameterize stable sextics. Thus the corresponding pairs in these cases are stable, and the forgetful map $\pi: \calP^{nu}\to \overline \calM\setminus\{\omega\}$ is a $\bP^2$-fibration (up to finite stabilizers) in a neighborhood of $Z_3\cup Z_4$. Furthermore, the pairs $(X,H)$ parameterized by $\pi^{-1}(Z_3\cup Z_4)$ are easily seen to be $\epsilon$-KSBA stable (see \S\ref{sectz4} and \S\ref{sectz3} for a discussion of the geometry of $(X,H)$ in these cases).

To handle the strictly semistable locus ($\overline Z_1\cup \overline Z_2$) we note the following:  if $C$ is semistable, there exists $1$-PS $\lambda$ adapted to $C$ (i.e. $\mu(C,\lambda)=0$), which then  singles out a  (possibly partial) flag $p_\lambda\in L_\lambda\subset \bP^2$. This flag is always in a special position with respect to $C$; typically $p_\lambda$ is a singular point of $C$ and $L_\lambda$ is a line of highest multiplicity in the tangent cone at $p_\lambda$. The stability of the pair $(C,L)$ (for $C$ strictly semistable) is typically determined by the position of $L$ with respect to the flag $p_\lambda\in L_\lambda$. In particular, we note
\begin{itemize}
\item[(a)] if $p_\lambda\not\in L$ for all $\lambda$ with $\mu(C,\lambda)=0$ (i.e. $L$ is generic), then $(C,L)$ is $\epsilon$-stable;
\item[(b)] if $L=L_\lambda$ (i.e. $L$ is very special), then $(C,L)$ is $\epsilon$-unstable.
\end{itemize}
While these two rules allow us to determine the stability/unstability of most pairs $(C,L)$, there are additional possibilities that might lead to pairs $(C,L)$ that are $\alpha$-strictly semi-stable for $\alpha$ varying in an interval.

The behavior of stability conditions for pairs over the strictly semistable locus  $\overline Z_1\cup \overline Z_2$ is analyzed by the following two propositions:
\begin{proposition}\label{proppairz2}
Assume that $C$ corresponds to a semi-stable orbit mapping to $\overline{Z}_2\subset \overline{\calM}$. Thus 
\begin{itemize}
\item[i)] either $C$ has a singularity at a point $p$ of multiplicity $4$, in which case if $p\in L$, then  $(C,L)$ is $\epsilon$-unstable;
\item[ii)] or $C$ contains a double line $L_0$, in which case if $L=L_0$ then 
$(C,L)$ is $\epsilon$-unstable. If $C$ contains a double line $L_0$ which is also tangent to the residual quartic in  a point $p$ and $p\in L$, but $L\neq L_0$, then $(C,L)$ is $\epsilon$-semistable with associated minimal orbit: 
$$\left(V(x_0^2x_2^2(x_0x_2-x_1^2)), V(x_1)\right).$$
\end{itemize}
If $(C,L)$ is not one of the three degenerate cases from above, then $(C,L)$ is $\epsilon$-stable.  Furthermore, in this case, the associated double cover is $\epsilon$-KSBA stable. 
\end{proposition}
\begin{proof} The Mumford numerical function for pairs is $\mu^\epsilon((C,L),\lambda)=\mu(C,\lambda)+\epsilon \mu(L,\lambda)$ (see also \cite[Sect. 2]{raduthesis}). In the first case, by considering the $1$-PS $\lambda$ with weights $(2,-1,-1)$, we get $\mu(C,\lambda)=0$ and then $\mu(L,\lambda)=-1$ if $p\in L$; thus, an unstable pair by the numerical criterion. Similarly, the second case follows by considering the $1$-PS of weights $(1,1,-2)$. The case of a double line tangent to a quartic is similar to Proposition \ref{proppairz1} below. Finally, if neither of these three degeneracy conditions are satisfied, we are in the situation (a) discussed above (i.e. generic line from the GIT point of view) and $(C,L)$ will be stable. For the geometric analysis of these cases see \S\ref{sectz2}. 
\end{proof}

\begin{proposition}\label{proppairz1}
Assume that $C$ corresponds to a minimal orbit of type $\overline Z_1\setminus \{\zeta,\omega\}$. Then, $C$ has one or two singular points $p$ of the following type: $p$ is a triple point with tangent consisting of a triple line $L_0$, and the singularity at $p$ is either $\Ee$ or $T_{2,3,r}$ ($r\ge 7$) or degenerate cusp of type $x^2(x+y^2)$ (i.e. double conic tangent to the residual conic). The following hold:
\begin{itemize}
\item[i)] if $p\not \in L$ (for both $p$  if  two special singularities), then $(C,L)$ is $\epsilon$-stable;
\item[ii)] $L=L_0$ (i.e. $L$ is the special direction through $p$), then $(C,L)$ is $\epsilon$-unstable;
\item[iii)] otherwise (i.e. $L$ passes through $p$, but it is not special), $(C,L)$ is $\epsilon$-semi-stable.   The minimal orbits in this case are given by 
\begin{itemize}
\item[(Case $Z_1$)] $\left(V((x_0x_2-a_1x_1^2)(x_0x_2-a_2x_1^2)(x_0x_2-a_3x_1^2)), V(x_1)\right)$, for  $a_i$ distinct; 
\item[(Case $\tau$)] $\left(V((x_0x_2-x_1^2)(x_0x_2-ax_1^2),V(x_1)\right)$, for $a\neq 1$.
\end{itemize} 
\end{itemize}
Furthermore, the double cover $(X,H)$ is $\epsilon$-KSBA stable iff $(C,L)$ is $\epsilon$-GIT stable. 
\end{proposition}
\begin{proof}
In this situation, the adapted $1$-PS $\lambda$ has weights $(1,0,-1)$. If $L$ passes through $p$, but not in a special direction, $\mu(L,\lambda)=0$, and the configuration remains semi-stable for all $0<\epsilon\ll 1$. Finally, to prove the equivalence of the two stability conditions it suffices to note that if $L$ passes through $p$ (regardless of $L=L_0$ or not)  the associated pair $(X,\epsilon H)$ is not slc (see also \S\ref{sectz1}). 
\end{proof}

The above discussion gives the non-unigonal case of Theorem \ref{thmgitpair}: 
\begin{corollary}\label{cornonunigonal}
Let $(C,L)$ be a pair consisting of a plane sextic and a line. Let $(X,H)$ be the associated double cover. Assume that 
\begin{itemize}
\item[i)] $(C,L)$ is $\epsilon$-GIT stable,
\item[ii)] and the orbit closure of  $C$ does not contain the triple conic.  
\end{itemize}
Then $(X,H)$ is $\epsilon$-MMP stable. Conversely, assume that $(X,H)$ is $\epsilon$-KSBA stable  and that $H$ is base-point-free. Then $(X,H)$ is the double cover associated to a pair $(C,L)$ satisfying the two conditions from above. 
\end{corollary}

%%%%%%%%%%%%%%%%%%%%%%%%%%%%%%%%%%%%%%
%%% VGIT for triple conic locus
\subsection{Blow-up of the triple conic locus}
Shah's construction of $\widehat \calM$ replaces the degenerations to the triple conic ($\omega\in\overline \calM$) by double covers of $\Sigma_4^0\subset \bP^5$, the cone over the rational normal curve  of degree $4$. Additionally, as noted in Theorem \ref{thmshah2}, all the semistable points in the GIT quotient $\widehat \calM$ correspond to degenerations of $K3$ surfaces with slc singularities. We now consider pairs consisting of such a surface (semistable double cover of $\bP^2$ or $\Sigma_4^0$) together with a hyperplane section. As explained, the goal here is to construct a neighborhood $U$ of the unigonal divisor which can be glued to $\calP_2^{nu}$  to give $\widehat \calP_2$.

To start, we note the following uniform description of the double covers of $\bP^2$ and $\Sigma_4^0$: they are complete intersections of the form
\begin{equation}\label{equniform}
\{z^2 - f_6(x_i, y) = f_2(x_i, y) =0 \} \subset \bP(1, 1, 1, 2, 3).
\end{equation}
The non-unigonal case corresponds to the situation 
$f_2(0, 0, 0, 1) \neq 0$. After a change of coordinates, one can normalize the equation in this case to: 
$$\{z^2 - f_6(x_i, y) = y =0 \} \subset \bP(1, 1, 1, 2, 3),$$
which leads to the usual case of double covers of $\bP^2$ branched along a sextic. Similarly, the double covers of $\Sigma_4^0$ correspond to the case when $f_2(0,0,0,1)=0$ and $f_2(x_1,x_2,x_3,0)$ has maximal rank.
Thus, one can choose the normal form:
$$\{z^2 - f_6(x_i, y) = x_0x_2 -x_1^2 =0 \} \subset \bP(1, 1, 1, 2, 3).$$
Moreover, as $f_6(0, 0, 0, 1) \neq0$ (i.e. the ramification curve does not pass through the vertex of $\Sigma_4^0$), we can further assume that $f_6(x_i, y) = y^3+y g_4(x_i) + g_{6}(x_i)$. 

Due to the fact that the automorphism group of the weighted projective space $\bP(1, 1, 1, 2, 3)$ is not reductive, a uniform GIT description would be difficult (see \cite[Sect. 4]{shah}). Instead, Shah \cite[Sect. 5]{shah} uses a local description of the GIT quotient $\overline \calM$ near the orbit $\omega$ of the triple conic and a gluing construction. The following is just a rephrasing of the main point of the  construction of Shah.
\begin{lemma}
Locally near $\omega$, $\overline \calM$ is identified with the (affine) quotient 
$$\left(\Sym^{12}V\times\Sym^8V\right)/\SL(2),$$
where $V$ is the standard $\SL(2)$ representation.  The Kirwan blow-up $\widehat \calM\to \overline \calM$ is modeled on the weighted blow-up of the origin in the vector space $\Sym^{12}V\times\Sym^8V$. In particular, the exceptional (unigonal) divisor  is identified with the GIT quotient 
$$\left(\bP \Sym^{12}V\times\bP\Sym^{8}V\right)\gquot_{\calO(3,2)} \SL(2).$$
\end{lemma}
\begin{proof}
 By definition, $\overline\calM\cong \bP\Sym^6W\gquot \SL(3)$ where $W$ is a standard representation of $\SL(3)$.  Luna's slice theorem  describe $\overline \calM$ locally at $\omega$ as the quotient of a normal slice to the orbit of a triple conic by the stabilizer $\SL(2)$. Since the conic is $\bP^1$ embedded by Veronese in $\bP^2$, we get an identification of $W=\Sym^2V$ as a $\SL(2)$-representation. Then the normal slice (as a $\SL(2)$-representation) is the summand $\Sym^{12} V\times \Sym^8V$ in $\Sym^6W\cong\Sym^6(\Sym^2V)$. The lemma follows. 

Alternatively, note  that in the normal form described above, the group preserving it is $\SL(2)$. Then, by viewing $x_i$ as sections of $\calO_{\bP^1}(2)$ we can identify $g_4(x_i)$ and $g_6(x_i)$ with binary forms $p_8(u, v)$ and $p_{12}(u, v)$ respectively.
 \end{proof}
 
 From the perspective of the lemma, we can view a line in $\bP^2$ (i.e. a section of the polarization) as an element of $\bP\Sym^2V$. We model the neighborhood $U$ of the unigonal divisor as the quotient 
\begin{equation}\label{defuni}
\left(\textrm{Bl}_0\left(\Sym^{12}V\times\Sym^8V\right)\times \bP\Sym^2V\right)\gquot\SL(2).\end{equation} 
The map $\widehat \calP_2\to \widehat \calM$ (locally near the unigonal divisor) is the induced map (at the level of quotients) by the first projection. By construction, $U$ glues to the non-unigonal quotient $\calP_2^{ns}$ to give $\widehat\calP_2$ together with a forgetful map $\widehat \calP_2\to \widehat \calM$. 
 
 To complete the proof of Theorem \ref{thmgitpair}, it remains to describe the stability conditions in the unigonal case. In other words, to study the quotient 
 $$\left(\bP \Sym^{12}V\times\bP\Sym^{8}V\times \bP\Sym^2V\right)\gquot_{\calO(3,2,\epsilon)} \SL(2).$$
To analyze this quotient in the geometric context of our paper, we rephrase a result of Shah (\cite[Theorem 4.3]{shah}) as follows:
\begin{theorem}[Shah]\label{thmunigonal}
With notation as above, let $B$ be a curve of $\Sigma_4^0$ given by $f_6(x_i,y)=y^3+y g_4(x_i) + g_{6}(x_i)$, and $X$ be the associated double cover of $\Sigma_4^0$. The following hold: 
\begin{itemize}
\item[(1)] If $B$ is stable and reduced, then $X$ has at most simple singularities.
\item[(2)] The minimal orbits of $B$ which are strictly semistable and reduced give surfaces $X$ with two $\Ee$ singularities. These minimal orbits  
 are parameterized by a (affine) rational curve $U_1\subset \widehat Z_1\setminus \widehat\tau$ (see Figure \ref{fig4}). 
\item[(3)] If $B$ is stable and non-reduced then $B=2C_1+C_2$ with $C_i$ rational normal curves that intersect transversely. These orbits are parameterized by a rational curve $U_2\subset \widehat Z_3\setminus \widehat\tau$.
\item[(4)] In addition to the cases given by (1), (2), (3), there is a single additional minimal orbit corresponding to the case $B=2C_1+C_2$ with $C_i$ rational normal curves that intersect tangentially at two points. This orbit maps to the point $\xi\in \widehat \tau$. 
\end{itemize}
Furthermore, the surfaces $X$ degenerating to  case (2), but not corresponding to minimal orbits, are rational with a unique $\widetilde E_8$ singularity. Similarly, the surfaces degenerating to case (4) have a singularity of type $T_{2,3,r}$ (for $r\ge 7$) or a degenerate cusp of type $z^2+x^2(x+y^2)$. 
\end{theorem}
\begin{proof}
As mentioned, this is precisely  \cite[Thm 4.3]{shah}. We only comment here on the singularities of the semi-stable objects. In case (2), the minimal orbits are given by 
$f_6(x_i,y)=y^3+ a_1y u^4v^4+ a_2u^6v^6$
(via the identification of $x_i$ with binary quadrics as above). In affine coordinates $X$ is given by $z^2=y^3+a_1yu^4+a_2u^6$, which is an $\Ee$ singularity. Note that the discriminant condition to get an $\Ee$ singularity and not worse coincides with the non-degeneracy condition from \cite[Thm 4.3]{shah}. By semicontinuity, one gets the same type of singularities for non-minimal orbits degenerating to  case (2) (N.B. there has to be at least one non-ADE singularity, otherwise $B$ would stable by (1); $\Ee$ deforms only to ADE singularities).

In cases $(3)$ and $(4)$, $B$ has the form $f= (y+\theta)^2( y-2 \theta) = y^3- 3 y \theta^2- 2 \theta^3$, where $\theta= p_4(u, v)$.  $B$ is stable iff  two divisors defined by $y+ \theta=0$ and $y-2 \theta=0$ intersect transversely in four distinct points. Similarly, the minimal orbit for the strictly semistable case corresponds to two double points. The singularity claim follows (see esp. \cite[\S16.2.9]{agzv1}). 
\end{proof}

We now conclude the analysis of stability in the unigonal case:\begin{corollary}\label{corunigonal}
Let $(B, L)$ be a pair defined as above and $(X, H)$ be the associated double cover. Then
\begin{itemize}
\item[(1)] $(B, L)$ is $\epsilon$-GIT stable iff $L$ doesn't pass through the $\Ee$, $T_{2,3,r}$, or degenerate cusps singularities. If   $(B, L)$ is $\epsilon$-GIT stable then $(X,H)$ is $\epsilon$-KSBA stable. Conversely, if $(X,H)$ is $\epsilon$-KSBA stable and $X$ is a double cover of $\Sigma_4^0$, then $(X,H)$ is induced by a $\epsilon$-GIT stable $(B, L)$ pair.  
\item[(2)] The minimal orbits of  $\epsilon$-strictly semistable pairs $(B, L)$ 
are given by $B$ with minimal orbit (as in Thm. \ref{thmunigonal}) and $L$ such that it passes through the  two $\Ee$ singularities if $B$ is reduced or through the two tangent points otherwise. 
\end{itemize}
\end{corollary}
\begin{proof}
The stability analysis is similar to that of Prop. \ref{proppairz1}. For the converse that  $\epsilon$-KSBA stable implies GIT stability, a closer look at the GIT stability conditions shows that unstable curves $B$ have singularities worse than simple elliptic, cusp, or degenerate cusp (compare Prop. \ref{propslc}).
\end{proof}

%%%%%%%%%%%%%%%%%%%%%%%%%%%%%%%%%%%%%%
%%% Sect5 : GIT vs. MMP
\section{Classification of slc stable pairs via GIT}\label{sectslclimit}
In Section \ref{sectreview} we have discussed the Shah's compactification $\widehat \calM$ which gives a compactification with weak geometric meaning for the moduli space $\calF_2$ of degree two $K3$ surfaces. In Section \ref{sectksba}, we have shown that considering pairs $(X,H)$ instead of polarized $K3$ surfaces $(X,\calO_X(H))$ gives a proper and separated moduli stack $\overline{\calP}_2$. Then, in section \ref{sectvgit}, via GIT, we have constructed an approximation $\widehat{\calP}_2$ of the space $\overline{\calP}_2$. Namely, there is a birational map $\widehat{\calP}_2\dashrightarrow \overline{\calP}_2$ which is an isomorphism over the stable locus $\calP_2^s\subset \widehat\calP_2$ (cf. Thm. \ref{thmgitpair}). We also recall that the strictly semistable locus $\widehat\calP_2\setminus \calP_2^s$ is a surface $\widetilde Z_1$ mapping  one-to-one to the stratum $\widehat Z_1\subset \widehat\calM$, and that the pairs parameterized (in the sense of GIT) by  $\widetilde Z_1$ are not KSBA stable.  To complete the description of $\overline{\calP}_2$, it remains to understand the KSBA replacement of the strictly semistable locus $\widetilde Z_1\subset\widehat{\calP}_2$. This is done in the following theorem.
\begin{theorem}\label{thmflip}
The birational map 
$$\overline{\calP}_2\dashrightarrow \widehat{\calP}_2$$
is a flip as in diagram \eqref{diagflip} that replaces the strictly semistable locus $\widetilde Z_1$ in $ \widehat{\calP}_2$ by locus of stable KSBA pairs $(X,H)$ of type $X=V_1\cup_EV_2$, where $V_i$  are degree $1$ del Pezzo surfaces or allowable degenerations of them (in $\bP(1,1,2,3)$) glued along an anticanonical section $E$ of $V_i$. Moreover, this transformation is compatible with the projection to $(\calD/\Gamma_2)^*$, and thus there is a natural forgetful map $\overline{\calP}_2\to (\calD/\Gamma_2)^*$.
\end{theorem}
\begin{proof}
By Thompson's Theorem \ref{thmthompson} (see also Section \ref{sectpolanti}, esp. Proposition \ref{classifydeg2}) and the GIT analysis (see esp. Corollaries \ref{cornonunigonal} and \ref{corunigonal}), we see that the only KSBA stable pairs $(X,H)$ that are not occurring in the GIT quotient $\widehat \calP_2$ are those for which $X$ is a union of two two del Pezzo surfaces of degree $1$ glued along an anticanonical section $E$. More precisely, $X$ is given as
$$X=V(z^2-f_6(x_i,y), x_0x_2)\subset \bP(1,1,1,2,3)$$
and $H$ is induced by a linear form in $x_i$. We denote by $\Delta_{2E_8}\subset \overline\calP_2$ the closure of this locus. We have $\dim \Delta_{2E_8}=19$ corresponding to $1$ modulus for $E$, $8$ moduli for each of the del Pezzo surfaces $V_i$, and $1$ modulus for each of the polarizing divisors $H_i\in|-K_{V_i}|$ on $V_i$. Note also that since $X$ has at worst slc singularities, which is the same as insignificant cohomological singularities, there is (at least set theoretically) a map $\overline\calP_2\to (\calD/\Gamma_2)^*$ to the Baily--Borel compactification of $\calF_2$ (see \cite{shahinsignificant}).  
From \cite[Thm. 5.4, \S(5.2.2)]{friedmanannals} (see also Section \ref{secttype2} below), we obtain that $\Delta_{2E_8}\subset \overline\calP_2$ maps to the closure $\overline{\mathrm{II}}_{2E_8+A_1}\cong \bP^1$ of the Type II component labeled by $2E_8+A_1$ (see Rem. \ref{labelbb} and Fig. \ref{fig4}). In other words, there is a fibration $\Delta_{2E_8}\to \bP^1$ given by the $j$-invariant of the gluing curve $E$ with $18$-dimensional fibers.

On the other hand, for the strictly semi-stable locus $\widetilde Z_1\subset \widehat \calP_2$, we have the morphism:
\begin{eqnarray*}
\widehat\calP_2&\to\widehat\calM&\to (\calD/\Gamma_2)^*\\
\widetilde Z_1&\to \widehat Z_1&\to\overline{\mathrm{II}}_{2E_8+A_1},
\end{eqnarray*}
which realizes $\widetilde Z_1$ as a $\bP^1$-fibration (up to finite stabilizer issues) over $\overline{\mathrm{II}}_{2E_8+A_1}\cong\bP^1$; the fibration is given again by a $j$-invariant (see Rem. \ref{remjinv}). Geometrically, the points of $\widetilde Z_1$ are in one-to-one correspondence with pairs $(X,H)$ where $X$ is the double cover of $\bP^2$ (or similarly for $\Sigma_4^0$) branched in the union three conics pairwise tangent at two fixed points, and $H$ is induced from the line 
passing through these two points. The surface $X$ will have two $\Ee$ singularities. Since $H$ passes through them, $(X,H)$ is KSBA unstable. The KSBA replacement (obtained by applying Thm. \ref{thmksba}) is analyzed in \S\ref{semistablereplace} below. Essentially, the resolution $V$ of $X$ is a non-minimal elliptic ruled surface (over some elliptic curve $E$) with two disjoint $(-1)$-sections. Then, the semi-stable model associated to such a surface is $X_0=V_1\cup_E V\cup_E V_2$ with $V_i$ degree $1$ del Pezzo with fixed anticanonical section $E$. The KSBA model contracts $V$ resulting in $V_1\cup_E  V_2$ which corresponds to a point in $\Delta_{2E_8}$. On the other hand, the GIT model contracts the surfaces $V_i$ giving the surface $X$ which corresponds to a point in $\widetilde Z_1$.  Thus, the birational map $\overline\calP_2\dashrightarrow \widehat\calP_2$ (defined over $(\calD/\Gamma)^*$)  replaces $\Delta_{2E_8}\subset \overline\calP_2$  by $\widetilde Z_1\subset \widehat \calP_2$ by forgetting the modulus of $V$ and $V_1\cup V_2$ respectively. 

In other words, we obtain the following diagram:
\begin{equation}\label{diagflip}
\xymatrix{
&\widetilde \calP\ar@{->}[rdd]\ar@{->}[ldd]&\\
&\Delta_{E_8^2+A_1}\ar@{^{(}->}[u]\ar@{->}[rdd]\ar@{->}[ldd]&\\
\overline{\calP}_2\ar[dr] |!{[ur];[d]}\hole&&\widehat{\calP}_2\ar[dl] |!{[ul];[d]}\hole\\
\Delta_{E_8^2}\ar@{^{(}->}[u]\ar@{->}[rd]&(\calD/\Gamma_2)^*&\widetilde Z_1\ar@{^{(}->}[u]\ar@{->}[ld]\\
&\overline{\mathrm{II}}_{2E_8+A_1}\ar@{^{(}->}[u]&
}
\end{equation}
where $\widetilde \calP$ is a blow-up of $\widehat \calP_2$ along $\widetilde Z_1$ with exceptional divisor $\Delta_{2E_8+A_1}$ which parameterizes pairs $(X_0,H)$ with $X_0=V_1\cup V_0\cup V_2$ and $H=H_1+F_0+H_2$ as discussed in \S\ref{semistablereplace}. At this point the description is only set theoretical. To see that maps in \eqref{diagflip} are actually morphisms, one can proceed in two ways: Geometrically, one can construct a neighborhood of $\Delta_{2E_8+A_1}$ by deformation theory as in \cite{friedmanannals} (see also Rem. \ref{remtoroidal}) and then glue it to common open subset $\calP_2^s$  of  $\widehat \calP_2$ and $\overline \calP_2$ to obtain $\widetilde \calP$. Alternatively, we define $\widetilde \calP$ as the Kirwan desingularization of $\widehat \calP$ along the strictly semi-stable locus $\widetilde Z_1$ (N.B. there are only $\bC^*$-stabilizers). Then, we obtain \eqref{diagflip} by the analysis of the GIT quotient $\widehat \calP$ along $\widetilde Z_1$ as discussed in \S\ref{kirwanp2} below. 
\end{proof}

\begin{remark}
The following degenerations of $V_1\cup_E V_2$ are allowed.  First, $E$ is either smooth elliptic (Type II case) or nodal irreducible with $p_a(E)=1$ (Type III case). The del Pezzo surfaces are allowed to have ADE singularities. As degenerate cases, we allow also cones, i.e. singularities of type $\Ee$  (the cone over a degree $1$ elliptic curve) in the Type II case or the degenerate cusp which is obtained as the cone over a nodal curve. The normalization in this latter case is in fact $\bP^2$ (see  Rem. \ref{mostdegdp}). 
\end{remark}

%%% Semi-stable replacement
\subsection{The KSBA replacement of the strictly semistable locus}\label{semistablereplace}
We are now interested in identifying the KSBA stable replacement for the strictly semistable locus $\widetilde Z_1\subset \widehat\calP_2$. As usually, we consider a family $(\sX,\calH)/\Delta$ of semi-stable GIT pairs such that the central fiber $(X,H)$ is strictly semi-stable (and thus $0\in \Delta$ maps to $\widetilde Z_1\subset \widehat\calP_2$). In fact, without loss of generality we can assume $(X,H)$ to correspond to a minimal orbit. As usually, to understand the KSBA limit for $\sX^*/\Delta^*$  one has to arrange $(\sX,\calH)/\Delta$ in a semi-stable (or even Kulikov) form and then follow the arguments of Theorem \ref{thmksba} to obtain the KSBA limit. We sketch the computation below. Note that the semi-stable computations here are ``generic'' and their role is to give a geometric interpretation for Theorem \ref{thmflip}. The global properties of diagram \eqref{diagflip} follow from the discussion of \S\ref{kirwanp2}. 

\subsubsection{The geometry of the minimal orbits of $\widehat \calP_2$}\label{resolve2e8}
Consider the pairs $(X;H)$ associated to a minimal orbit of a strictly semistable point (i.e. corresponding to a point in $\widetilde Z_1$). As discussed, $X$ is the double cover of  $\bP^2$ branched along the sextic (the unigonal case is similar and left to the reader)
$$C=V\left((x_0x_2-\alpha_1 x_1^2)(x_0x_2-\alpha_2 x_1^2)(x_0x_2-\alpha_3 
x_1^2)\right)\subset \bP^2, \textrm{ for some } \alpha_1,\alpha_2,\alpha_3\in \bC.$$
The polarization $H$ is the pull-back of the line $L=V(x_1)$. The surface $X$ has two $\widetilde E_8$ singularities corresponding to the points $p_1=(1:0:0)$ and $p_2=(0:0:1)$. Consider $V\to X$ the minimal resolution (obtained via a single weighted blow-up of $p_i$). It is well known that  the two exceptional divisors $E_1$ and $E_2$ are elliptic curves with $E_1^2=E_2^2=-1$. 

We are interested here in understanding the geometry of the pair $(V;H)$, where $H$ is the polarizing divisor (i.e. $|H|$ defines the map $V\to X\to \bP^2$, and $H$ maps to $L\subset \bP^2$).  Using the standard procedure of resolving double covers,  we obtain the following commutative diagram:
$$\xymatrix{
V\ar@{->}[rr]\ar@{->}[dd]\ar@{->}[rd]&&\widetilde P\ar@{->}[dd]\ar@{->}[rd]&\\ 
&E\ar'[r][rr]&& \Pi\cong \bP^1\\
X\ar@{->}[rr]\ar@{-->}[ur]&&\bP^2\ar@{-->}[ur]&
}$$
where
\begin{itemize}
\item $\Pi\cong \bP^1$ is the pencil of conics passing through the points  $p_1,p_2\in \bP^2$ and tangent to the lines $V(x_0)$ and $V(x_2)$;
\item $E\to \Pi$ is the double cover branched in the points corresponding to the three special conics $V(x_0x_2-\alpha_ix_1^2)$ and to the double line $V(x_1^2)$ (in coordinates, $E$ is the double cover of $\bP^1$ branched in $\alpha_1,\alpha_2,\alpha_3,\infty$);
\item $\widetilde P$ is the blow-up of $\bP^2$ twice at each of the points  $p_1$ and $p_2$, followed by the contraction of the resulting two $(-2)$-curves (alternatively, $\widetilde P$ is a weighted blow-up of $\bP^2$ at the two special points; $\widetilde P$ has two $A_1$ singularities); 
\item the horizontal arrows are double covers;
\item $\bP^2\dashrightarrow \Pi$ is the conic bundle fibration given by mapping a  point $x(\neq p_i)\in \bP^2$ to the unique member of the pencil $\Pi$ that passes through $x$.
\end{itemize}

\begin{lemma}
With notations as above, $V\to E$  is an elliptic ruled surface, and the two 
exceptional divisors $E_1$ and $E_2$ are two disjoint sections of self-intersection $-1$. Furthermore, 
\begin{itemize} 
\item[i)] The strict transform of the line $L=V(x_1)\subset \bP^2$ gives a special fiber $F_0$, which can be taken as the origin of $E$ (and of the sections $E_i$).
\item[ii)] There are two reducible fibers for $V\to E$ corresponding to the reducible conic $V(x_0x_2)$ in the pencil $\Pi$. In particular, $V$ is the blow-up at two points of a geometrically ruled surface. 
\item[iii)] The pullback of the line $L\subset \bP^2$ to $V$ is 
$$H=F_0+E_1+E_2,$$ and thus the linear system $|F_0+E_1+E_2|$ gives the map $ V\to X\xrightarrow{2:1}\bP^2$. 
\end{itemize}
\end{lemma}
\begin{proof}The claims follow easily from the above discussion.\end{proof}
\begin{remark}[compare Rem. \ref{remjinv}]
The cross-ratio associated to the elliptic curve $E$ is $\lambda=\frac{\alpha_1-\alpha_3}{\alpha_2-\alpha_3}$ and then the $j$-invariant is
$
j(E)=2^8\cdot \frac{(\lambda^2-\lambda+1)^3}{\lambda^2(\lambda-1)^2}$.
Alternatively, the affine equation of $X$ near a singular point 
$
z^2=(x-\alpha_1y^2)(x-\alpha_2y^2)(x-\alpha_3y^2)$
can be put into the Weierstrass form $z^2=x^3+Axy^4+By^6$,  where
$$
A=\sigma_2-\frac{(\sigma_1)^2}{3},\ B=\sigma_3+\frac{\sigma_1\sigma_2}{3}-\frac{2(\sigma_1)^3}{27},$$
and $\sigma_i$ are the elementary symmetric functions in $\alpha_i$. 
In this form, the discriminant and the $j$-invariant have the following expressions 
$$
\Delta=4A^3+27B^2,\ j=1728\cdot\frac{4A^3}{27B^2}.$$
Note that $\Delta=-(\alpha_1-\alpha_2)^2(\alpha_2-\alpha_3)^2(\alpha_3-\alpha_1)^2$, and thus the elliptic curve $E$ is singular iff two of the $\alpha_i$ coincide (i.e.  case $\tau$). In this case, if not all three coincide, it is easy to see that the resulting surface is a Type III degeneration. 
\end{remark}

\subsubsection{The semi-stable reduction}
We are now interested in understanding the Kulikov model associated to a $1$-parameter family with central fiber $(X,H)$ as above. In the generic case, the weighted blow-up of the two $\Ee$ singularities of $X$ (as above, but this time keeping track of the ambient threefold $\sX$) gives a semi-stable model with central fiber $X_0=V_1\cup V\cup V_2$, where 

\begin{itemize}
\item $V$ is the resolution of $X$ as in \S\ref{resolve2e8}, and $V_1$ and $V_2$ are degree $1$ del Pezzo surfaces with a marked anticanonical section $D_i\cong E$; 
\item $V_1$ is glued to $V$ along the elliptic curve $D_0\cong E\cong E_1$ (and similarly for $V_2$); the gluing is such that the unique base point of $|-K_{V_1}|$ matches with the point $E_1\cap F_0$. 
\end{itemize} 
Note that the triple point formulas (e.g.  $D_1^2+E_1^2=1+(-1)=0$) are satisfied, and we have indeed a Kulikov model. 

Keeping track of the polarization, we obtain for the GIT model the polarized components: 
$$(V_1,D_1;0),\ (V,E_1+E_2;F_0+E_1+E_2),\ (V_2,D_2,0).$$
This leads to a KSBA unstable limit since $(V,E_1+E_2;\epsilon(F_0+E_1+E_2))$ is not slc (the polarizing divisor contains a double curve). Then the KSBA replacement is obtained by twisting the polarization by $\calO(-V)$ (on the total threefold space), resulting into the polarized components
$$(V_1,D_1;H_1),\ (V,E_1+E_2;F_0),\ (V_2,D_2;H_2),$$
with $H_i\in|-K_{V_i}|$. In this model, $V$ becomes a $1$-surface and thus will be contacted to the elliptic curve $E$. We conclude that the KSBA limit in this situation is $\overline{X}_0=V_1\cup_EV_2$, two  degree $1$ del Pezzos glued along an anticanonical surface. 

\begin{remark}
Note that the same analysis can be easily extended to cover the Type III case (i.e. $E$ nodal) as well. For example in this case, the central fiber  $X_0=V_1\cup V\cup V_2$ of the corresponding semi-stable degeneration  will contain again of two del Pezzo surfaces $V_i$, but the double curves will be nodal anticanonical sections. Additionally, $V$ is a rational surface which is a degeneration of an elliptic ruled surface. Namely,  $V$ is a non-normal surface such that its normalization $\widetilde V$ is a ruled surface blown-up at two points. For instance, $V$ can be obtained by blowing-up two points on a section of self-intersection $1$ on the Hirzebruch surface $\bF_1$, and then gluing two irreducible fibers of the resulting surface to get $V$.
\end{remark}

%%% Kirwan blow-up for P_2
\subsection{The structure of the Kirwan blow-up $\widetilde \calP$ of $\widehat\calP_2$ along the strictly semi-stable locus}\label{kirwanp2}
We now analyze the structure of the Kirwan blow-up $\widetilde \calP$ of $\widehat\calP_2$ along the strictly semi-stable locus $\widetilde Z_1$. Namely, we show that $\widetilde \calP\to \widehat\calP_2$ is a fibration over $\widetilde Z_1$ with fiber isomorphic to $W\bP^9\times W\bP^9$, which can be identified with the moduli of polarized surfaces of type $V_1\cup_E V_2$ with $V_i$ degree $1$ del Pezzo (and $E$ fixed).  On the other hand $\widetilde Z_1$ is a $\bP^1$ fibration over  $\overline {\mathrm{II}}_{2E_8+A_1}$. Combining this with the geometric analysis of \S\ref{semistablereplace}, we  get the diagram \eqref{diagflip}, completing the proof of Theorem \ref{thmflip}.  
\subsubsection{Preliminaries on degree $1$ del Pezzo surfaces}
We recall that a degree $1$ del Pezzo surface has the following anticanonical model:
\begin{equation}\label{eqdelpezzo}
\left\{z^2=y^3+yg_4(x_0,x_1)+g_6(x_0,x_1)\right\}\subset \bP(1,1,2,3),
\end{equation}
and the point $(0:0:1:1)$ is base locus of the anticanonical linear system. We are interested in surfaces of type $V_1\cup_EV_2$, where both $V_i$ are degree $1$ del Pezzo surfaces, $E$ is anticanonical section, and the gluing is such that the base points of the anticanonical systems on $V_i$ match (a necessary condition for $V_1\cup_EV_2$ to occur as a central fiber in a degree $2$ $K3$ degeneration). As already noted, $V_1\cup_EV_2$ has the following description:
$$V_1\cup_EV_2=\left\{z^2=y^3+yg_4(x_0,x_1,x_2)+g_6(x_0,x_1,x_2), x_0x_2=0\right\}\subset \bP(1,1,1,2,3),$$
which is compatible with \eqref{equniform} and Theorem \ref{thmthompson}. 
Note that the gluing curve is  given by intersecting with $V(x_0,x_2)$:
$$E=\{z^2=y^3+Byx_1^4+Cx_1^6\}\subset \bP(1,2,3),$$
an elliptic curve in Weierstrass form. The polarizing divisor in this situation is given by a linear form $l(x_0,x_1,x_2)$. Finally, the cone over $E$ is given by 
$$\{z^2=y^3+Byx_1^4+Cx_1^6\}\subset \bP(1,1,2,3),$$
with vertex at $(1:0:0:0)\in \bP(1,1,2,3)$; it appears when $x_0$ (or $x_2$) does not occur in $g_4$ and $g_6$. 
\subsubsection{The moduli of degree $1$ del Pezzo with a marked anticanonical section}
A theorem of Pinkham and Looijenga (e.g. \cite{pinkham} and \cite{looijengainv,looijengasimple}) identifies the moduli space of pairs $(V,D)$ consisting of a degree $1$ del Pezzo surface $V$ with a marked hyperplane section $D$ to the weighted projective space $\bP(1,2,2,3,3,4,4,5,6)$ (N.B. $2,2,\dots,5,6$ are  the coefficients of the simple roots $\alpha_i$ of $E_8$ in the highest root $\widetilde \alpha$). One way of seeing this is to consider the versal deformation of an $\Ee$ singularity $\{z^2=y^3+Byx^4+Cx^6\}\subset (\bC^3,0)$,  which is given by 
\begin{equation}\label{eqversal}
\left\{z^2=y^3+Byx^4+Cx^6+t_1yx^3+\dots+t_5x^5+\dots+t_{10}\right\}\subset (\bC^3,0)\times (\bC^{10},0).
\end{equation}
Since this is a singularity with  $\Gm$-action in the sense of Pinkham, one gets a $\bC^*$-action on the germ $(t_i)\in (\bC^{10},0)$ with weights $0$, $-1$, $-2$, $-2$, $-3$, $-3$, $-4$, $-4$, $-5$, $-6$. The deformations in the $0$-direction correspond to equisingular deformations (i.e. keep the $\Ee$ singularity, but modify the $j$-invariant). The deformations in the negative weight correspond to smoothing deformations, and when considered modulo $\bC^*$ one gets that the resulting quotient $\bP(1,2,2,3,3,4,4,5,6)$ is a moduli space of del Pezzo pairs  $(V,D)$ with $D$ isomorphic to the fixed elliptic curve $E=V(z^2-y^3+Byx^4+Cx^6)$. Simply, this corresponds to homogenizing the equation of the versal deformation of $\Ee$; the result is  the del Pezzo equation \eqref{eqdelpezzo} (here $x=\frac{x_1}{x_0}$ and the section $D$ corresponds to the hyperplane at infinity $V(x_0)$). 

\begin{remark} \label{remstructure2e8a1}
Alternatively, the moduli of pairs $(V,D)$ with $D\cong E$ (a fixed elliptic curve) can be obtained by  considering the mixed Hodge structure (MHS) on $V\setminus D$. Since $V$ is the blow-up of $\bP^2$ at $8$ points lying on $E$, the classifying space for these MHS is $E\otimes_\bZ E_8$. Then, the moduli space of pairs $(V,D)$ is 
$$(E\otimes_\bZ E_8)/W(E_8)\cong\bP(1,2,2,3,3,4,4,5,6),$$
the isomorphism to the weighted projective space being the content (in more general circumstances) of the above mentioned Theorem of Looijenga (\cite{looijengainv}). This description allows us to see the moduli of semi-stable models $X_0=V_1\cup_EV\cup_EV_2$ (for $E$ fixed) as the product $\bP(1,2,2,3,3,4,4,5,6)\times \bP^1\times \bP(1,2,2,3,3,4,4,5,6)$ (by applying Looijenga's theorem to the root lattice $R=2E_8+A_1$; see also Rem. \ref{remtoroidal}). 
\end{remark}

Here, we are interested in the moduli of triples $(V,D;H)$ where $V$ is a degree $1$ del Pezzo, $D\cong E$ is a fixed anticanonical section, and $H$ is another anticanonical section. By a simple modification of the  arguments from above for $(V,D)$, we get:
\begin{lemma}\label{modulipairs}
The moduli of triples $(V,D;H)$, where $V$ is a degree $1$ del Pezzo, $D,H\in |-K_V|$, $D\neq H$, and $D\cong E$ is fixed, is the quasi-projective variety 
$$\bP(1,1,2,2,3,3,4,4,5,6)\setminus\{(1:0:\dots:0)\}.$$  
Furthermore, it has a $1$-point compactification to the weighted projective space $\bP^9(1,\dots,6)$. The extra point naturally corresponds to the cone over $E$ (with an $\Ee$ singularity at the vertex $v$) together with a hyperplane section away $v$. 
\end{lemma}
\begin{proof}
We obtain this, by considering as before the negative weight deformations  of $\Ee$. We define $H$ to be the hyperplane $\{x=a\}$ in affine coordinates for $a\in \bC$ (or equivalently $\{x_1=ax_0\}$ after homogenization). As before we obtain the affine quotient $(\bC\times \bC^9)/\bC^*$, which gives the weighted projective space from the theorem. Note that if the $\bC^9$ component is non-zero we obtain a unique pair $(V,D)$ with $V$ a degree $1$ del Pezzo (with at worst ADE singularities) and $D\cong E$. From the $\bC$ component of the parameter space, we also get a hyperplane section $H\in |-K_V|\setminus \{D\}$ (N.B. $H=D$ corresponds to $a=\infty$). Finally, if the $\bC^9$ component vanishes, we must have $a\neq 1$ and then we get the cone over $E$ together with a hyperplane section not passing through the vertex $v=(1:0:0:0)$ of the cone. 
\end{proof}

\begin{remark}
The weighted blow-up of $(1:0:\dots:0)\in \bP^9(1,1,2,\dots,6)$ will give a $\bP^1$-fibration over $\bP^8(1,2,\dots,6)$. This corresponds geometrically to triples $(V,D;H)$ with $D\cong E$ fixed and $H$ moving in the linear system $H\in|-K_V|\cong \bP^1$ with no restriction on $H$.  Thus, the difference to the moduli space of the lemma is that all triples $(V,D;H)$ with $D=H$ are replaced in \ref{modulipairs} by the cone over $E$ (plus a general hyperplane section). This is the correct moduli space from the KSBA perspective (see also  \S\ref{sectz1} esp.  case $2E_8+A_1$(C) and Fig. \ref{fig2e8b}). 
\end{remark}
We conclude
\begin{corollary}\label{cor2dp1}
The moduli of pairs $(X,H)$ with $X=V_1\cup_EV_2$, where $V_i$ are degree $1$ del Pezzo or degenerations (in $\bP(1,1,2,3)$), $H_{\mid V_i}\in|-K_{V_i}|$, and such that $E$ is fixed and $(X,\epsilon H)$ is slc is   
$$\bP(1,1,2,2,3,3,4,4,5,6)\times \bP(1,1,2,2,3,3,4,4,5,6).$$
\end{corollary}
\begin{proof}
This follows from the previous lemma, by noting that $(X,H)$ is uniquely determined by $(V_i,D_i;H_i)$ (where $H_i=H_{\mid V_i}$ and $D_i\cong E$). 
\end{proof}

\subsubsection{The structure of $\widehat \calP_2$ near $\widetilde Z_1$}
Let $x=(c,l)\in \bP^{N}\times \bP^2$ (where $\bP^N$ is the Hilbert scheme of sextics) be a point corresponding to the minimal orbit $(C,L)$  given by $C=V((x_0x_2-\alpha_1 x_1^2)(x_0x_2-\alpha_2 x_1^2)(x_0x_2-\alpha_3 
x_1^2))$, and $L=V(x_1)$. We are interested in the structure of the quotient $\widehat\calP_2$ near the projection $\bar x\in \widetilde Z_1$ of $x$. 

As before, by Luna's slice theorem a local model is given by the normal slice $\calN_x$ to the orbit $G\cdot x$. It is immediate to see that the stabilizer group $G_x$ acts on the space of sextics $T_x\bP^N$ with weights
$$
\begin{matrix}
\textrm{Weight}&\pm 6&\pm 5&\pm 4&\pm 3&\pm 2&\pm 1&0\\
\textrm{Multiplicity}&1&1&2&2&3&3&3
\end{matrix}
$$
and on the space of linear forms $T_l\bP^2$ by weights $\pm1$. Since $G\cdot x=G/G_x$, we get that the action of $G_x$ on  the tangent space to the orbit $T_x(G\cdot x)$ by weights $\pm 2$, $\pm 1$, and $0$ with multiplicities $1$, $2$, and $1$ respectively. We conclude that  $G_x\cong \bC^*$ acts on $N_x\cong \bC^{22}$ by  
$$
\begin{matrix}
\textrm{Weight}&\pm 6&\pm 5&\pm 4&\pm 3&\pm 2&\pm 1&0\\
\textrm{Multiplicity}&1&1&2&2&2&2&2
\end{matrix}
$$
Thus locally (in etale topology) near $\bar x$, $\widehat\calP_2$ is the quotient of $\bC^{22}$ by $\bC^*$ with weights given as above. The two dimensional $0$-weight direction corresponds to deformations preserving the strictly minimal orbit. We conclude:
\begin{lemma}
The fiber of the Kirwan blow-up $\widetilde \calP\to \widehat\calP_2$ over $\bar x\in \widetilde Z_1$ is 
$$\bP(1,1,2,2,3,3,4,4,5,6)\times \bP(1,1,2,2,3,3,4,4,5,6).$$
\qed
\end{lemma}

Using this lemma and the $\bP^1$-fibration of $\widetilde Z_1$ given by the $j$-invariant (see Rem. \ref{remjinv}), we obtain that the exceptional divisor $\Delta_{2E_8+A_1}$ of $\widetilde \calP$ has a fibration over $\bP^1$ (the compactified $j$-line) with fiber $W\bP^9\times W\bP^9\times \bP^1$ (compare Rem. \ref{remstructure2e8a1}). Geometrically, these fibers parameterize surfaces of type $V_1\cup_E V\cup_E V_2$ as described in \S\ref{semistablereplace}. The projection  $\widetilde \calP\to \overline\calP_2$ is then given by the contraction of the $\bP^1$ direction. For further discussion of the geometry in this case see \S\ref{sectz1} (esp. Fig. \ref{fig2e8b}).

%%%%%%%%%%%%%%%%%%%%%%%%%%%%%%%%%%%%%%
%%% Part 3 - Detailed classification
%%%%%%%%%%%%%%%%%%%%%%%%%%%%%%%%%%%%%%
%%% Sect6 : Classification of Type II Degenerations

\section{Classification of Type II Degenerations}\label{secttype2}
As established above, the moduli space of stable pairs $\overline{\calP}_2$ maps to the Baily--Borel $(\calD/\Gamma_2)^*$, which is generically a $\bP^2$-bundle. In this section we discuss the structure of the boundary of $\overline{\calP}_2$ over the Type II boundary in $(\calD/\Gamma_2)^*$. 

We recall that the semi-stable model in a Type II degeneration is a chain of surfaces $X_0=V_0\cup \dots\cup V_r$ (with $V_i$ meeting $V_{i+1}$) such that 
\begin{itemize}
\item $V_0$ and $V_r$ are rational surfaces;
\item $V_i$ are elliptic ruled surfaces;
\item the double curves are smooth elliptic and isomorphic to a fixed  curve $E$;
\item the double curves are anticanonical sections for $V_0$ and $V_r$ and sections for $V_i$ for $i\neq 0,r$. 
\end{itemize}
The normalization $X^\nu$ of the central fiber $X=\overline X_0$ of a relative log canonical model will have at most two simple elliptic singularities (besides ADE singularities), and the double curve of the normalization will be either empty (if $X$ is normal with simple elliptic singularities) or give disjoint elliptic curves, all isomorphic to $E$.  

Associated to every Type II degeneration of $K3$ surfaces there are two basic invariants: a continuos invariant, the modulus of $E$ (possibly with a level structure), and a discrete invariant, the isometry class of the lattice $\Gr^W_2H^2(X_0)_{\textrm{prim.}}$ (N.B. it is defined over $\bZ$, see \cite[\S3]{friedmanannals} for details). The discrete invariant determines the Type II boundary component to which a one-parameter semistable degeneration with central fiber $X_0$ would map. The continuous invariant determines the actual point in the Type II component where the degeneration maps (recall that the Type II components are quotients of $\fH$ by modular groups). 

For degree two, there are four Type II Baily--Borel boundary components,  labeled by the root lattices $A_{17}$, $E_7+D_{10}$, $D_{16}+A_1$, and $2E_8+ A_1$. The geometric meaning of these components (via GIT) was explained in Section \ref{sectreview} (see esp. Figure \ref{fig4}). Furthermore, Friedman \cite[Thm. 5.4]{friedmanannals} has classified the semistable models corresponding to these four cases, subject to the following normalization assumptions: there are only two components for $X_0$ (i.e.  a union of two rational surfaces), and the polarization meets the double curve (i.e. $L_i.D_i>0$) (N.B. any Type II degeneration can be arranged to satisfy these conditions, cf. \cite[Thm. 2.2]{friedmanannals}). 

As we will see below, the Friedman semi-stable models can be used to understand all the Type II boundary pairs in $\overline \calP_2$. However, as the proof of Theorem \ref{thmksba} shows, one needs to allow two operations on Friedman's models: base changes (introducing elliptic ruled surfaces in the middle) and twists by components $V_i$ of $X_0$ (these have the effect of modifying the polarization on $(V_i,D_i)$ from $L_i$ to $L_i-D_i$). Combining the list of polarized semi-stable models of Friedman with the GIT analysis of Section \ref{sectvgit}, one gets a clear picture of $\overline{\calP}_2\to (\calD/\Gamma_2)^*$ over the Type II strata. We discuss this below. The discussion can be summarized as follows:

\begin{theorem}\label{thmtype2}
The preimage in $\overline \calP_2$ of the  Type II boundary in $(\calD/\Gamma_2)^*$ consists of six irreducible components $\II_i$ as summarized in Table \ref{table1} (the index $i$ corresponds to the case in the table). Furthermore, via $\overline \calP_2\to(\calD/\Gamma_2)^*$
\begin{itemize} 
\item[i)] $\II_1$ maps to $\II_{A_{17}}$  (see Prop. \ref{poverz4});
\item[ii)] $\II_2$ and $\II_5$ map to $\II_{2E_8+A_1}$ (see Prop. \ref{poverz1});
\item[iii)] $\II_3$ maps to $\II_{D_{16}+A_1}$  (see Prop. \ref{poverz3});
\item[ii)] $\II_4$ and $\II_6$ map to $\II_{E_7+D_{10}}$ (see Prop. \ref{poverz2}).
\end{itemize}
\end{theorem}
\begin{proof}
As discussed, the stable limits of degenerations of $K3$ surfaces are essentially determined by the components that are $0$-surfaces. The polarized $0$-surfaces in degree $2$ are classified by Proposition \ref{classifydeg2}; this gives the $6$ cases of Table \ref{table1}. The fact that these cases occur and that there is exactly one boundary component associated to each case follows from the GIT analysis. A detailed discussion of the GIT models and of the connection to the abstract point of view is done in Propositions \ref{poverz4}, \ref{poverz2}, \ref{poverz3}, and \ref{poverz1} below.
\end{proof}

\begin{remark}
The analysis of the Type II boundary of $\overline{\calP}_2$ does not depend (after ignoring finite quotient issues) on the $j$-invariant associated to the corresponding geometric object (compare Remark \ref{remjinv}).  Thus, the pre-images of Type II components in $\overline\calP_2$ will be certain fibrations over the affine $j$-line. The limits as $j\to \infty$ give Type III pairs; the classification of those will be discussed in Section \ref{secttype3}. 
\end{remark}

%%% Z_4 (A_17)
\subsection{Case $A_{17}$ (\cite[(5.2.1)]{friedmanannals}, \cite[Thm. 2.4 (II.4)]{shah},  \cite[Table 1 (II.3)]{thompson})}\label{sectz4} 
The Type II Baily-Borel boundary component $\textrm{II}_{A_{17}}$ corresponds to the stratum $Z_4$ in the GIT quotient $\widehat{\calM}$, and in fact $\widehat{\calM}\to(\calD/\Gamma_2)^*$ is an isomorphism along this stratum. Since the stratum $Z_4$ corresponds to stable GIT points, it follows that $\widehat \calP_2\to \widehat \calM$ is a $\bP^2$-fibration (up to finite stabilizers) over $Z_4\cong \bA^1$. Finally, $\overline\calP_2$ and $\widehat \calP_2$ agree over this stratum. Thus, we conclude (see \S\ref{sectexample} for the last statement): 

\begin{proposition}\label{poverz4}
The moduli of pairs $\overline\calP$ is (up to finite stabilizers) a $\bP^2$-fibration over the Type II boundary component $\II_{A_{17}}$. In fact, the closure of this locus remains a $\bP^2$-bundle over $\overline{\II}_{A_{17}}\cong \bP^1$.
\end{proposition}

\subsubsection{GIT Model} The underlying surface $\overline{X}_0$ is 
$$z^2=f_3(x_i)^2$$ 
for a smooth plane cubic $f_3$. The normalization of $\overline{X}_0$ is two copies of $\bP^2$ with the double curve being the elliptic curve $E=V(f_3)$. Since the plane sextic $f_3^2$ is stable and we are working with $0<\epsilon\ll 1$ linearization, the choice of polarizing divisor is irrelevant here. The same is true from the KSBA perspective, as the polarizing divisor (a line in each copy of $\bP^2$) cannot have a common component with the double curve.

\subsubsection{Friedman's Model} The semi-stable surface $X_0=V_1\cup V_2$ is obtained  by gluing two copies of $\bP^2$ along an elliptic curve, and then blowing-up $18$ times  along the elliptic curve. The polarization is the pullback to $V_i$ of a line in $\bP^2$. Thus, the relative log canonical model $\overline{X}_0$ is obtained by contracting all these $(-1)$-curves, and coincides to the GIT model. Note also that in this situation the polarizations $L_i$ on the components $(V_i,D_i)$ satisfy $L_i^2=1$ and $L_i.D_i=3$. Thus, no twist is possible (compare Lemma \ref{lemtwist}).

%%% Z_2 (E_7+D_10)
\subsection{Case $E_7+D_{10}$,   (\cite[(5.2.3)]{friedmanannals}, \cite[Thm. 2.4 (II.2)]{shah}, \cite[Table 1 (II.0h), (II.1)]{thompson})}\label{sectz2}
 In this case, the corresponding GIT stratum is $Z_2$. Again $\widehat\calM$ and $(\calD/\Gamma_2)^*$ agree over this stratum. Also, $\widehat \calP_2$ and $\overline \calP_2$ agree over the preimage of this stratum. However, in contrast to the previous case, $\widehat \calP_2\to \widehat\calM$ is not a $\bP^2$-bundle over $Z_2$. Here the choice of polarizing divisor instead of line bundle is essential: without a divisor one only gets strictly semistable points, in contrast when the divisor is considered all the pairs are either stable or unstable (in a GIT sense, but this coincides with the KSBA stability here). The analysis of the models associated to this stratum, gives the following result:

\begin{proposition}\label{poverz2}
The fiber over a point of the boundary component $\II_{E_7+D_{10}}$ consists of the closure of the following two strata:
\begin{itemize}
\item[(A)] a $9$-dimensional stratum parameterizing triples $(V,D;H)$ where $V$ is a degree $2$ del Pezzo surface, $D,H\in |-K_V|$, $D\cong E$ is a fixed elliptic curve, and $H\neq D$. 
\item[(B)] a $12$-dimensional component  parameterizing surfaces that are double covers of $\bP^2$ branched in a reduced sextic with an $\Es$ singularity and a hyperplane section not passing through the singularity;
\end{itemize}
The closure of each of these components is obtained by adding a rational curve, which is common to both. The gluing curve parameterizes pairs 
$(X,H)$  with $X$ a minimal elliptic ruled surface with a section of self-intersection $2$ and another section collapsed to a $\Es$ singularity, and a divisor $H\in |\sigma+2f|$ (where $\sigma$ is the class of the $(-2)$ section, and $f$ the class of a fiber). 
\end{proposition}

\subsubsection{GIT Model} The minimal orbits corresponding to points of $Z_2$ are given by $x_0^2f_4(x_1,x_2)$ for a binary quartic with distinct roots. One distinguishes three distinct geometric possibilities:
\begin{itemize}
\item[(A)] a sextic containing a double line: the normalized double cover will be a degree two del Pezzo, and the line will give (as the double curve of the normalization) the anticanonical section $D$;
\item[(B)] a reduced sextic with a unique  $\widetilde E_7$ singularity;
\item[(C)] a sextic with both a $\widetilde E_7$ and  a double line. \end{itemize}

When considering additionally a hyperplane section (i.e. passing from $\widehat{\calM}$ to $\widehat{\calP}$), the orbits become separated as in the toy example of \S\ref{sectexample}.  As discussed in Proposition \ref{proppairz2}, the restrictions for the hyperplane section are not to pass through the $\Es$ singularity or to coincide with the double line. A simple dimension count (for a fixed $j$-invariant) gives the dimensions of the proposition. For example, the moduli of degree $2$ del Pezzo surfaces containing a fixed elliptic curve $E$ is $7$-dimensional and isomorphic to the weighted projective space  $\left(E\otimes_\bZ E_7\right)/W(E_7)\cong \bP^7(1,1,2,2,2,3,3,4)$.  The choice of polarizing divisor gives $2$ additional dimensions. 

Finally, case (C) is a specialization of both (A) and (B). The double cover associated to the sextic of case (C) is, after normalization and resolution of the $\Es$ singularity, a minimal elliptic ruled surface with marked sections of self-intersection $2$ and $-2$. Because, of the $\bC^*$ stabilizer in case (C), there is (up to isomorphism) only a $1$-dimensional choice for the  hyperplane section $H$. Explicitly, $X$ is the normalization of the double cover $z^2=x_0^2f_4(x_1,x_2)$ and $H$ is the pullback of the line $L=V(x_0+bx_1+cx_2)$ in $\bP^2$ and the modulus is given by $(b:c)\in \bP^1$. 

\subsubsection{Friedman's Model} The semistable model in this case $X_0=V_1\cup V_2$ gives after the contraction of the $(-1)$-curves orthogonal to the polarization the following two relatively minimal polarized anticanonical surfaces: $(\oV_1,\oD_1;\oL_1)$  a degree $2$ del Pezzo with $\oD_1,\oL_1\in |-K_{\oV_1}|$, and $(\oV_2,\oD_2;\oL_2)\cong (\bF_1, 2\sigma+4f;f)$. Note that to get the semi-stable model $X_0$ starting from $\oV_1\cup \oV_2$ one needs $10$ additional blowups, but these are irrelevant from the perspective of the relative log canonical model. Also note, that in this model $V_1$ is $0$-surface, and $V_2$ is a $1$-surface. Thus, the central fiber of the relative log canonical model will be just $\oV_1$ with $\oD_1$ marked. This is precisely  case (A) from above. 

\begin{figure}[htb!]
\includegraphics[scale=0.7]{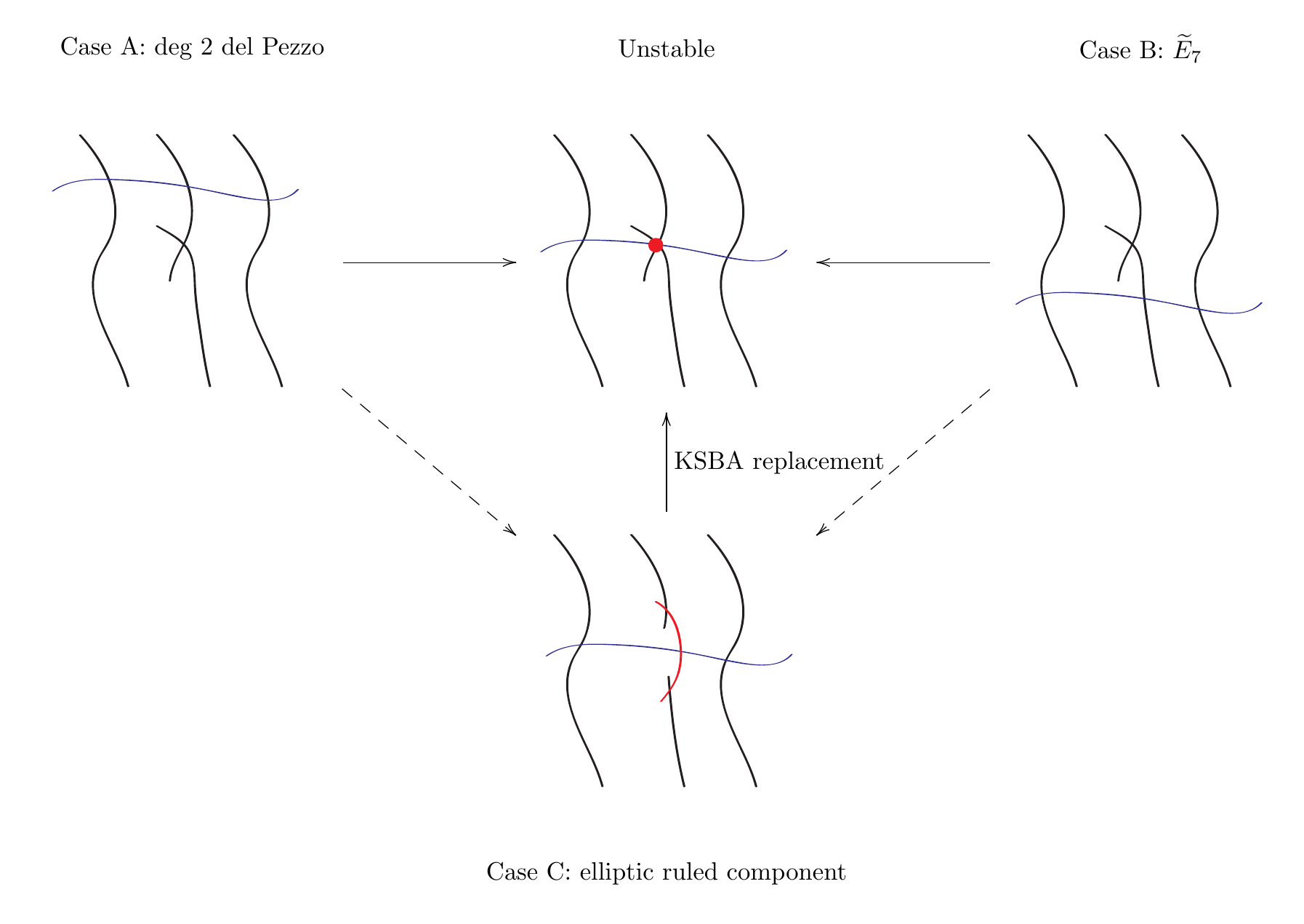} 
\caption{The degenerations in the $E_7+D_{10}$ case}\label{figtwist}
\end{figure}

It is easy to understand how the other models occur:  when trying to compactify the moduli of pairs, one has to replace the case $\oH_1=\oD_1$ (the polarizing divisor coincides with the double curve) by a slc model.  This is achieved by replacing the semistable model $V_1\cup V_2$ by a model $V_1\cup V\cup V_2$ where $V$ is a minimal elliptic ruled surface glued along two sections: $D_{1}$ of self-intersection $-2$ and $D_{2}$ of self-intersection $2$. When $H_1=D_1$ occurs, one applies the twist, and moves the polarization from $V_1$ to $V$ (see Figure \ref{figtwist}).  

Finally, note that the semistable model $X_0=V_1\cup V\cup V_2$ dominates all the GIT models (A), (B), and (C). The difference between the three cases is given by the lift of the polarization $\calL^*$ in a one parameter degeneration $\sX/\Delta$ from the generic fiber to the family.  Explicitly, encoding the lift $\calL$ by the degrees on the components of $X_0$, we get
\begin{itemize}
\item $(2,0,0)$ corresponds to (A) (degree $2$ del Pezzo);
\item $(0,2,0)$ corresponds to (C) (elliptic ruled with a section collapsed to an $\Es$ singularity);
\item $(0,0,2)$ corresponds to (B) (rational with an $\Es$ singularity). 
\end{itemize}
Note that the lifts of $\calL$ are related by the twist operation.

\begin{remark}[$\overline{\calP}_d$ vs. partial toroidal compactifications]\label{remtoroidal}
The deformation theory of a semistable model $X_0$ is well understood:  one can obtain a partial compactification of the moduli of $K3$ surfaces by adding a divisor parameterizing semistable models of a fixed combinatorial type and satisfying the {\it $d$-semi-stablity condition} of Friedman (see  \cite[\S4]{friedmanannals}, \cite{friedmansmooth}, \cite{kn}, and \cite{olsson}). On the other hand, for Type II boundary components (for Type III see \cite{friedmanscattone}), it is known that there exists a unique toroidal compactification over a Type II boundary stratum in $(\calD/\Gamma_d)^*$. One can show that the geometric divisor obtained by local trivial deformations of semistable models and the toroidal divisor can be identified via the theory of degenerations of Hodge structures (this is essentially the content of \cite{friedmanannals}). However, the issue is that, in polarized case, there are several possible liftings of the polarization. Thus one obtains several boundary divisors ($3$ in the example above) which are not distinguishable at the level of Hodge theory. The choice of polarizing divisor (giving $\bP^g$-bundles over these boundary divisors) allows one to glue the various boundary divisors. In the $\bP^g$-bundle $\overline{\calP}_d$ over $\calF_d$ these divisors are contracted to smaller dimensional components.

For instance, in the $E_7+D_{10}$ case discussed above, the deformation theory for $X_0=V_1\cup V\cup V_2$ will give a boundary divisor which is essentially a $W\bP^7\times W\bP^{10}$ fibration (coming from $(E_7\otimes E)/W(E_7)$ and $(D_{10}\otimes E)/W(D_{10})$ respectively) over the $j$-line. The choice of the lifting the polarization gives $3$ copies of this divisor, say $\Delta_{(2,0,0)}$, $\Delta_{(0,2,0)}$, and $\Delta_{(0,0,2)}$ (N.B. these can be viewed as divisors in a partial non-separated compactification of $\calF_2$, compare \cite{olsson}). The choice of a polarizing divisor $H$ gives $\bP^2$-bundles, say $\widetilde\Delta_{(2,0,0)}\to \Delta_{(2,0,0)}$, over each of these copies (N.B. $\widetilde\Delta_{(2,0,0)}$ can be viewed as divisors in a partial compactification of $\calP_2$). Then, the case $H_1=D_1$ can be viewed as giving a gluing of the copy $\widetilde\Delta_{(2,0,0)}$ with the $\widetilde\Delta_{(0,2,0)}$ copy; and similar gluing for the divisor corresponding to $(0,2,0)$ and $(0,0,2)$. Thus at the level of pairs, it is possible to give a partial toroidal like compactification for $\calP_2$ by adding the divisors $\widetilde\Delta_{(*,*,*)}$. Finally, in $\overline{\calP}_d$ these divisors will be collapsed to three smaller dimensional strata (e.g. in the $(2,0,0)$ case the $\bP^2$-bundle $\widetilde\Delta_{(0,2,0)}$ over $W\bP^{7}\times W\bP^{10}$ will be collapsed to a $\bP^2$-bundle over $W\bP^{7}$, giving the $9$-dimensional stratum (A)). 
\end{remark}

%%% Z_3 (D_16+A_1)
\subsection{Case $D_{16}+A_1$ (\cite[(5.2.4)]{friedmanannals},  \cite[Thm. 2.4 (II.2)]{shah}, \cite[Table 1 (II.2)]{thompson})}\label{sectz3}
 This case is quite similar to the $A_{17}$ case: both $Z_3$ (this case) and $Z_4$ correspond to  stable GIT loci. We note first that $\widehat \calM\to (\calD/\Gamma)^*$ is a $\bP^1$-bundle over the component $\II_{D_{16}+A_1}$. Specifically, $  \widehat Z_3\setminus\hat\tau \to \II_{D_{16}+A_1}\cong \bA^1$ is a $\bP^1$-bundle, the map being given by the $j$-invariant (compare Remark \ref{remjinv}).  This corresponds to the following geometric fact:

\begin{lemma}\label{lemmad16}
The choice of point in $\widehat Z_3\setminus\hat\tau$ corresponds to the choice of an elliptic normal curve of degree $4$ in $\bP^3$ together with a quadric containing it. 
\end{lemma}

The points of $\widehat Z_3\setminus\hat\tau$ are stable GIT points, by construction it follows that $\widehat\calP_2\to \widehat\calM$ is a $\bP^2$-bundle over this locus. Finally, $\overline \calP_2$ and $\widehat\calP_2$ agree here (the preimage of $\widehat Z_3\setminus\hat\tau$ is away from the flip locus). We conclude: 
\begin{proposition}\label{poverz3}
Over the component $\II_{D_{16}+A_1}$, $\overline \calP_2\to (\calD/\Gamma)^*$ is a $\bP^2\times \bP^1$-bundle.
\end{proposition}

\subsubsection{The GIT Model} The equation of the sextic corresponding to this case is $q_0^2q$, with the conditions that $q_0$ is smooth, $q$ is reduced, and $q_0$ and $q$ intersect transversally. The normalization $V_1$ of the double cover $z^2=q_0^2q$ is the quadric surface $z^2=q$ in $\bP^3$. The double curve of the normalization will be elliptic curve $D_1$ which   is the double cover of the conic $V(q_0)$ branched at the $4$ intersection points. A similar picture holds also  in the unigonal case (i.e. $U_3\subset \widehat Z_3$). This concludes the proof of the Lemma \ref{lemmad16}. Note that in this case all the points are stable, thus this stratum is modular even without the choice of a divisor.

Finally, the hyperplane section is the pullback of a line in $\bP^2$. It is a hyperplane section of the quadric $V_1$ but it is not an arbitrary section, in fact it lies in a two dimensional  linear subsystem characterized by the property that it cuts the elliptic curve $D_1$ in two conjugate points. This somewhat surprising fact is explained by the analysis of the semi-stable model below. 

\subsubsection{Friedman's Model} The relatively minimal models of the two components in this case are: $(\oV_1,\oD_1;\oH_1)=(\bF_0, 2f_1+2f_2;f_1+f_2)$, $(\oV_2,\oD_2;\oH_2)=(\bF_1,2\sigma+4f;2f)$. Note that $H_1^2=2$, $H_2^2=0$, $H_1.D_1=H_2.D_2=4$. Since $H_1.D_1>H_1^2$, it follows that no twisting is possible. This means that in contrast to the $E_7+D_{10}$ case there is only one model. Note that the condition on the hyperplane section noted in the previous paragraph (i.e. $H_1$ cuts two conjugate pairs of points on the elliptic double curve $D_1$) is imposed by the requirement of extending the polarization to the second component (even though this component is a $1$-surface, which is contracted to the double curve in the log canonical model).

Abstractly, this case corresponds to the case of polarized anticanonical pairs $(V,D;L)$ with $L^2=2$ and $L.D=4$. From Proposition \ref{propclassify}, we know that $V$ has to be a scroll and in fact a quadric surface in $\bP^3$. The results of Harbourne (e.g. Thm. \ref{thmldn0}) say that the linear system $|L|$ is base point free. Thus, the occurrence of the unigonal case might seem contradictory. The resolution of this apparent contradiction was given above: the allowed polarizing divisors are members of a linear subsystem of $|L|$.  

%%% Z_1 (2E_8+A_1)
\subsection{Case $2E_8+ A_1$ (\cite[(5.2.2)]{friedmanannals},  \cite[Thm. 2.4 (II.1)]{shah},  \cite[Table 1 (II.0h)]{thompson},  \cite[Table 2 (II.0u)]{thompson},  \cite[Table 2 (II.4)]{thompson})}\label{sectz1} 
This case is the most involved one. Specifically, on the GIT side this corresponds to the stratum $\widehat Z_1$, parameterizing curves with $\Ee$ singularities; these are strictly semi-stable points. When we consider the polarizing divisor, the orbits become stable if the divisor doesn't pass through the $\Ee$ singularity. If it passes through the singularity, we obtain a strictly semistable object, which will be replaced by the flip discussed in Section \ref{sectslclimit} by the case of two components (which both of them have to be del Pezzo of degree 1). We conclude:
\begin{proposition}\label{poverz1}
The fiber in $\overline\calP_2$ over a point in   $\II_{2E_8+A_1}\subset (\calD/\Gamma)^*$ consists of two components:
\begin{itemize}
\item[(A)] A component of dimension $18$ parameterizing $(X,H)$ with $X=V_1\cup_E V_2$, with both $V_i$ being degree $1$ del Pezzo surfaces glued along an elliptic curve $E$ (such that  the base points $p_i\in E$ of the anticanonical systems are matched).  This case can further degenerate to cases (C) and (D) (where one or both of $V_i$ degenerate to elliptic ruled surfaces with $\Ee$ singularities).  
\item[(B)] A component of dimension $11$ parameterizing rational surfaces with $\Ee$ singularities together with hyperplane sections not passing through the singularities. This can further degenerate to case (E) (i.e. elliptic ruled surfaces with two $\Ee$ singularities). 
\end{itemize}
The two components are glued along the $9$-dimensional (closure of the) stratum (C) (see Figure \ref{fig2e8b}). 
\end{proposition}

\begin{figure}[htb!]
\includegraphics[scale=0.65]{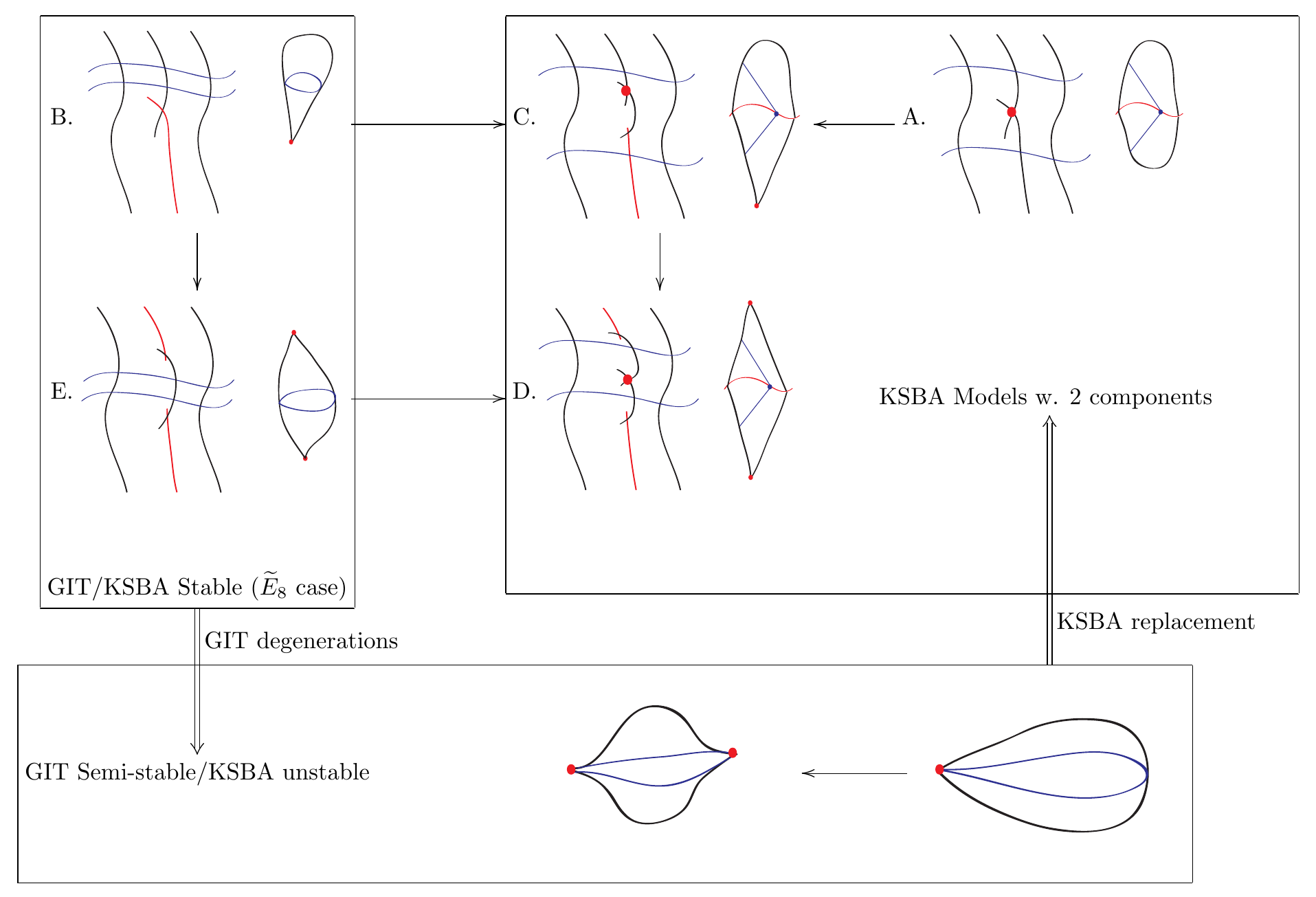} 
\caption{The degenerations in the $2E_{8}+A_1$ case}\label{fig2e8b}
\end{figure}

\begin{remark}
As already discussed in Section  \ref{sectslclimit}, one can be very precise about the structure of the various strata occurring in the proposition. For instance, the closure of the stratum (A) is the product of two weighted projective spaces $\bP^9(1,1,\dots,5)\times \bP^9(1,1,\dots,5)$ (see Cor. \ref{cor2dp1}).
\end{remark}

\subsubsection{GIT Model} Here we have several models. First, we have plane sextic curves with a unique $\widetilde E_8$ (depending on $10$ moduli, one of which is the $j$-invariant) and sextics with two $\widetilde E_8$ (depending on $2$ moduli) which are obtained from $q_1q_2 q_3$ with common axis. In the classification above, these cases correspond to (B) and (E). In  case (E), via a partial smoothing one obtains case (B) (see Lem. \ref{partialsmooth}). 

Next, we consider additionally the hyperplane section. If the hyperplane does not pass through the $\widetilde E_8$ singularity, then the resulting pair is both GIT and KSBA stable. If the hyperplane passes through the singularity, the pair is GIT semistable and slc unstable. By applying a semi-stable reduction as discussed in Section \ref{sectslclimit}, one obtains the case of two components which (unless are cones) have to be degree $1$ del Pezzo surfaces.

\subsubsection{Friedman's Model} Each of the two relatively minimal models $(\oV_i,\oD_i;\oH_i)$  of the components of \cite[(5.2.2)]{friedmanannals}  are degree $1$ del Pezzo surfaces with $\oD_i, \oH_i\in |-K_{\oV_i}|$. 
 This gives case (A) discussed above. As in the $E_7+D_{10}$ case, additional models can be obtained by base change and twisting. For instance, starting with  case (A), one needs to blow-up two more times to get a semi-stable model. Then applying a twist gives case (B): a single rational component $V_1$, which    is the blow-up of $10$ points on a cubic in $\bP^2$. Here, $H_1^2=2$ (and $H_2^2=0$), $H_1.D_1=0$, $D_1^2=-1$. By the results discussed in Section \ref{sectpolanti}, such a case either leads to a double cover of $\bP^2$ branched along a sextic with an $\widetilde E_8$ singularity if there is no fixed component, or to a unigonal type case (which corresponds to the $U_1\subset \widehat Z_1$ stratum in the GIT model). Finally, applying base changes to Friedman's model followed by twists leads to  cases (C), (D), and (E). It is interesting to note that all cases discussed in  \cite[p. 21]{thompson} occur in the $2E_8+A_1$ situation.

%%%%%%%%%%%%%%%%%%%%%%%%%%%%%%%%%%%%%%
%%% Sect7 : Type III 

\section{Classification of Type III Degenerations}\label{secttype3}
 We now discuss the case of Type III degenerations. According to the classification given by Proposition \ref{classifydeg2}, every $0$-surface $(V,D;L)$ that occurs in a Type III degeneration has a partial smoothing to a Type II Case $(V',D';L')$, i.e. as polarized surfaces $(V,L)$ and $(V',L')$ are deformation equivalent, $D'$ is a smoothing of the cycle of rational curves $D$. Thus,

\begin{theorem}\label{thmtype3}
The Type III locus in $\overline{\calP}_2$ is the closure of the Type II locus, in the sense of taking the closure (as $j\to \infty$) of the fibrations  over the Type II  Baily--Borel boundary components in $(\calD/\Gamma_2)^*$. In particular, there are $6$  Type III boundary components $\III_i$ in  $\overline{\calP}_2$ of dimensions $2$, $18$, $3$, $9$, $11$, and $12$ as described in Table \ref{table1}. Each of these components is irreducible except $\III_3$ which splits into two irreducible components $\III_\gamma$ and $\III_\delta$.  The incidence relations of the Type III components is summarized in Table \ref{table2}.
\end{theorem}

\begin{remark}
We note that the statement that all $0$-surfaces of Type III have a deformation to a Type II polarized anticanonical pair is only true for low degrees. For instance, the surfaces $\bF_n$ (for $n\ge 3$) carry an effective anticanonical divisor of Type III, but not one of Type II. Thus, for large degrees, at least a priori, there might be degenerations of Type III that are not limits of Type II degenerations.
\end{remark}

We will denote the six Type III boundary components by $\textrm{III}_i$ for $i\in \{1,\dots,6\}$ according to Table \ref{table1} and Proposition \ref{classifydeg2}. The generic point of each of these components was already described. Also, their basic structure is similar to that of their Type II counterparts (see  Propositions \ref{poverz4}, \ref{poverz2}, \ref{poverz3}, and \ref{poverz1}). The only significant difference is that the gluing of the Type III strata is more involved, reflecting the fact that it is easier for the polarizing divisor to pass through a log canonical center (i.e. $H$ might pass through a triple point, or  contain a component of the anticanonical cycle vs. $H$ has to  contain the anticanonical curve in the Type II case). A summary of the strata resulting from incidence of several Type III boundary components $\textrm{III}_i$ is given in Table \ref{table2} below. Note that the components $\textrm{III}_i$ for $i\neq 3$ are irreducible, but for $i=3$ we have a decomposition in irreducible components $\textrm{III}_3=\textrm{III}_\gamma\cup \textrm{III}_\delta$. We include $\textrm{III}_\gamma$ and $\textrm{III}_\delta$ in the table as they are included in some other components $\textrm{III}_i$. 

\begin{table}[htb!]
\renewcommand{\arraystretch}{1.4}
\begin{tabular}{|l|l|l|l|}
\hline
&Description (generic point)& dim& Contained in\\
\hline\hline
$\alpha$&$X=V_0\cup_DV_1$, $V_i\cong \bP^2$, $D$ nodal cubic&1&$\textrm{III}_1$, $\textrm{III}_\beta$, $\textrm{III}_\delta$\\
\hline
$\beta$& $X^\nu$ is a degree $2$ del Pezzo with an $A_1$ at $p$, $p\in D$&8& $\textrm{III}_4$, $\textrm{III}_6$\\
\hline
$\gamma$& $X^\nu$ is quadric in $\bP^3$, $D$ is the union of two conics&3& comp. of $\textrm{III}_3$, $\textrm{III}_\beta$\\
\hline
$\delta$& $X^\nu$ is a quadric in $\bP^3$, $D$ is a nodal quartic&3& comp. of $\textrm{III}_3$, $\textrm{III}_5$\\
\hline
$\epsilon$& $X^\nu$  quadric, $D=C\cup L_1\cup L_2$, $C$ conic, $L_i$ line&2&  $\textrm{III}_{\gamma}$, $\textrm{III}_\delta$\\
\hline
$\phi$&$X=V_0\cup_DV_1$, $V_0$ deg 1  dPezzo, $V_1^\nu\cong\bP^2$, $D$ nodal&9&$\textrm{III}_2$, $\textrm{III}_5$\\
\hline
$\zeta'$&$X=V_0\cup_DV_1$,  $V_i^\nu\cong\bP^2$, $D$ nodal&0&$\textrm{III}_\phi$\\
\hline
$\zeta$&$X=V_0\cup_DV_1$, $V_i\cong \bP^2$, $D$ is a triangle&0&$\textrm{III}_\alpha$, $\textrm{III}_\epsilon$\\
\hline
\end{tabular}
\vspace{0.2cm}
\caption{The incidence of Type III boundary components in $\overline \calP_2$}\label{table2}
\end{table}

\begin{remark}
As for Table \ref{table1}, when describing the stable $(X,H)$ corresponding to the generic point of a boundary stratum we ignore the polarizing divisor $H$. The dimension is the dimension of the stratum in $\overline\calP_2$ and thus takes into account $H$. Note that sometimes $X$ has positive dimensional stabilizer, leading to strata of dimension less than $2$ (compare Rem. \ref{remdef}). Finally, $D$ refers to a cycle of rational curves which is an anticanonical divisor on the normalized components of $X$. In some of the cases $D$ passes through a canonical singularity of (the normalization of) $X$;  a resolution of the singularity will bring $X$ and $D$ in a standard form. 
\end{remark}

To begin, we note that the gluing of Type III strata will reflect the structure of the boundary in the GIT quotient $\widehat \calM$ (see Figure \ref{fig4}) and 
the gluing of the type II components in $\overline{\calP}_2$. Namely, we recall that in the GIT quotient $\widehat \calM$, the Type III stratum is mapped to the rational curve $\hat\tau\cup \zeta$. The point $\zeta$ corresponds to the gluing of all strata. Similarly, the affine curve $\hat \tau$ corresponds to the gluing of the strata corresponding to $E_8^2+A_1$ and $D_{16}+A_1$. At the level of Type II pairs, the only gluing occurs for $E_7+D_{10}$ cases (A) and (B) (corresponding to $\textrm{III}_4$ and $\textrm{III}_6$) and for $E_8^2+A_1$ cases (A) and (B) (corresponding to $\textrm{III}_2$ and $\textrm{III}_5$). Using this information, we now identify for each Type III component $\textrm{III}_i$ some substrata along which the given component glues to other Type III components.

\begin{remark}
In general, given a boundary pair $(X,H)$ in $\calP_d$, further degenerations of it will have larger or equal number of components. It follows that the two boundary components $\III_1$ and $\III_2$ that parameterizes degenerations $X=V_1\cup_EV_2$ with two components are disjoint. These two boundary components will meet the other boundary components along $\III_\alpha$ and $\III_\zeta\in\III_\alpha $ for $\III_1$ and along $\III_\phi$ and   $\III_{\zeta'}\in\III_\phi$ for $\III_2$.  The points $\III_\zeta$ and $\III_{\zeta'}$ are in a certain sense the deepest degenerations for degree $2$ $K3$ pairs. They correspond to the so called {\it pillow degenerations} of $K3$ surfaces (e.g. \cite{pillow}), i.e. unions of $\bP^2$'s (polarized by $\calO(1)$) glued accordingly to the combinatorics of a triangulation of $S^2$ with $d$ triangles (see also Rem. \ref{mostdegdp}). 
An on-going project of Gross--Hacking--Keel (e.g. \cite{ghk}) investigates the deformations of such pillow surfaces (in all degrees) and thus (in particular) describes neighborhoods of $\III_\zeta$ and $\III_{\zeta'}$ in $\overline\calP_2$.  
\end{remark}

\subsection{$A_{17}$ case}
As already mentioned in Proposition \ref{poverz4}, the closure of the Type II locus in this case is still a $\bP^2$-bundle over $\overline{\II}_{A_{17}}\cong \bP^1$. The corresponding Type III stratum is $\III_1$ and is two dimensional. It parameterizes stable pairs $(X,H)$ where $X$ is the union of two $\bP^2$ glued along a nodal cubic $C$, $H$ corresponds to the choice of a line not passing through the nodes of $C$ (compare \S\ref{sectexample}). The generic case is $C$ is irreducible nodal; the stable pairs in this case belong only to the component $\III_1$. The component $\textrm{III}_1$ will be glued to other components along the locus where $C$ becomes reducible. We denote by $\textrm{III}_\alpha\subset \III_{1}$ the rational curve corresponding to $C$ reducible and by $\textrm{III}_\zeta\in\textrm{III}_\alpha $   the point corresponding to  $C$ becoming a triangle. 

Note that the entire component $\textrm{III}_1\subset \overline \calP_2$ maps to the point $\zeta$ in $\widehat \calM$ (via $\overline \calP_2\dashrightarrow \widehat \calP_2\to \widehat \calM$; N.B. the $\textrm{III}_1$ locus is not affected by the flip of Section \ref{sectslclimit}). Note also that the point $\textrm{III}_\zeta\in \overline \calP_2$ corresponds to the minimal orbit associated to $\zeta\in \widehat \calM$. We reiterate here that the main point is that at the level of pairs the moduli functor is separated and thus $\textrm{III}_\zeta$ corresponds to a single geometric object, in contrast to the point $\zeta$ which hides several orbits. 

\subsection{$E_7+D_{10}$ case} 
As already discussed, in this case there are two geometric possibilities: 
\begin{itemize}
\item[($\textrm{III}_4$)] Degenerations of $E_7+D_{10}$ (A): the elliptic section of the degree $2$ del Pezzo becomes nodal, or, in terms of sextics, a quartic plus a tangent line to it  (the line being counted with multiplicity $2$).
\item[($\textrm{III}_6$)] Degenerations of $E_7+D_{10}$ (B): the $\Es$ singularity degenerate to a cusp singularity $T_{2,4,5}$ and then further to other cusps  of type $T_{2,q,r}$ with $q\ge 4, r\ge 5$.
\end{itemize} 

For a fixed invariant $j$, the Type II components $E_7+D_{10}$ (A) and (B) 
are glued along a curve (see Prop. \ref{poverz2}). The Type III limit (i.e. $j\to \infty$) of this curve in $\overline \calP_2$ is the point $\textrm{III}_\zeta$. Thus, we have $\textrm{III}_\zeta\subseteq \textrm{III}_4\cap \textrm{III}_6$. However, it is immediate to see that the intersection of these two Type III components is larger. Namely, the maximal stratum which is a common degeneration of both the degree $2$ del Pezzo and $T_{2,4,5}$ cases corresponds to the double cover of $\bP^2$ branched in a nodal quartic with a double line passing through it. From the del Pezzo perspective, this would be a nodal degree $2$ del Pezzo with a hyperplane section through it (as an anticanonical section). From the perspective of cusp singularities, this is a degenerate cusp singularity, which has a partial smoothing to cusp singularities of type $T_{2,4,r}$. We call this stratum $\mathrm{III}_\beta$. Note that $\mathrm{III}_\alpha\subset \mathrm{III}_\beta$ (and then $\mathrm{III}_\zeta\subset \mathrm{III}_\alpha\subset \mathrm{III}_\beta$). 

We note that there is another special stratum (which is contained in $\textrm{III}_\beta$) in this case: the case of  the double cover $X$ of $\bP^2$ branched along  two double lines plus a generic quadric. The intersection of the two double lines leads to a degenerate cusp singularity, which has a partial smoothing to $T_{2,q,r}$ with $q,r\ge 5$. The normalization of $X$ will be a quadric in $\bP^3$  and the double curves of the normalization will be the union of two hyperplane sections (i.e. $(1,1)$ curves in the case $X^\nu\cong \bP^1\times \bP^1$). We call this stratum $\textrm{III}_{\gamma}$.

\subsection{$D_{16}+A_1$ case} We recall that the type II degenerations corresponding to this case are surfaces $X$ such that their normalizations are  quadrics in $\bP^3$ and such that the double curve $E$ is an elliptic quartic curve in $\bP^3$. These are obtained by considering double covers of type $z^2=q_0^2q$ of $\bP^2$ (where $q_0$ and $q$ are two conics). There are two distinct Type III degenerations (in codimension $1$) in this situation: either $V(q_0)$ becomes singular (i.e. union of two lines) or the two conics become tangent. The first case was labeled as $\textrm{III}_{\gamma}$ above. We call the second case $\textrm{III}_{\delta}$. In other words, the Type III component in this case is reducible:
 $$\textrm{III}_3=\textrm{III}_{\gamma}\cup \textrm{III}_{\delta}.$$
The two components intersect in the stratum corresponding to the double cover of $\bP^2$ branched in two double lines together with a conic which is tangent to one of the lines. Equivalently, at the level of quadrics in $\bP^3$, this corresponds to the case that the anticanonical divisor $D$ splits as $L_1\cup L_2\cup C$ with $L_1$ of type $(1,0)$, $L_2$ of type $(0,1)$ and $C$ of type $(1,1)$. We call this case $\textrm{III}_\epsilon$. 

\subsection{$2E_8+A_1$ case} Here again we have two possibilities:
\begin{itemize}
\item[($\textrm{III}_2$)] Degenerations of $2E_8+A_1$ (A): the two degree $1$ del Pezzo surfaces are glued along a nodal section. This can degenerate to the case when one or both of the del Pezzo surfaces become cones over this nodal curve.  We denote this case by $\textrm{III}_\phi$  (a degeneration of the Type II case  $2E_8+A_1$ (C)). 
\item[($\textrm{III}_5$)] Degenerations of $2E_8+A_1$ (B): the $\Ee$ becomes a $T_{2,3,7}$ or  worse singularity.  In the closure,  we can obtain case $\textrm{III}_\phi$ as a replacement of the case the polarizing divisor passes through the $T_{2,3,7}$ singularity (compare Prop. \ref{poverz1}), or case $\textrm{III}_\delta$  (a double conic plus another conic tangent to it) which is a degenerate cusp singularity that has a partial smoothing to $T_{2,3,7}$. 
\end{itemize}

\begin{remark}\label{mostdegdp}
As already noted, a degree $1$ del Pezzo can degenerate to the cone over an elliptic curve of degree $1$ giving an $\Ee$ singularity (the Type II case $2E_8+A_1$ (C)). This can further degenerate to the cone $\overline V$ over an irreducible nodal curve $C$ of arithmetic genus $1$ (case $\textrm{III}_\phi$ from above). We note that the normalization $V$ of $\overline V$  is  $\bP^2$. In fact,  $\overline V$ is obtained from $V\cong \bP^2$ by gluing together two lines. The nodal curve $C$ is the image of a third line (forming a triangle $D$) in $\bP^2$. In other words, $(\overline V,C)$ (or the normalized pair $(V,D)$) regarded as a $0$-component of a degree two degeneration fits into the classification given by Proposition \ref{classifydeg2}. However, the gluing in the limit surface $X$ is somewhat surprising and hints to the difficulty of an analogous classification for larger degrees. Finally, we note that the two surfaces corresponding to $\III_\zeta$ and $\III_{\zeta'}$ are obtained by gluing $2$ copies of $\bP^2$ with a marked triangle in each according to the two possible triangulations of $S^2$ with $2$ triangles (see also \cite{thurston}, \cite{triangulationk3}). 
\end{remark}

%%%%%%%%%%%%%%%%%%%%%%%%%%%%%%%%%%%%%%
\bibliography{refk3}
\end{document}